\documentclass{article}

\PassOptionsToPackage{numbers, compress}{natbib}
\usepackage{amsmath,amssymb,amsthm,fullpage,mathrsfs,pgf,tikz,caption,subcaption,mathtools}
\usepackage[ruled,vlined,linesnumbered]{algorithm2e}
\usepackage{natbib}
\usepackage{amsmath,amssymb,amsthm,mathtools}
\usepackage[utf8]{inputenc} 
\usepackage[T1]{fontenc}    
\usepackage{hyperref}       
\usepackage{url}            
\usepackage{booktabs}       
\usepackage{amsfonts}       
\usepackage{nicefrac}       
\usepackage{microtype}      
\usepackage[margin=1in]{geometry}
\usepackage{bbm}
\usepackage{enumitem}
\usepackage{xcolor}
\usepackage{cleveref}
\usepackage{mathabx}

\hypersetup{colorlinks=true,citecolor=blue,linkcolor=red}

\let\oldnl\nl
\newcommand{\nonl}{\renewcommand{\nl}{\let\nl\oldnl}}

\newtheorem{theorem}{Theorem}[section]
\newtheorem{corollary}[theorem]{Corollary}
\newtheorem{proposition}[theorem]{Proposition}
\newtheorem{lemma}[theorem]{Lemma}
\newtheorem{definition}[theorem]{Definition}

\newtheorem{claim}[theorem]{Claim}

\newcommand{\R}{\mathbb{R}}

\newcommand{\id}[1]{\mathbbm{1}_{#1}}

\newcommand\norm[1]{\left\lVert#1\right\rVert}

\title{Eigenstripping, Spectral Decay, and Edge-Expansion on Posets}

\author{%
Jason Gaitonde\thanks{Department of Computer Science, Cornell University. Email: \texttt{jsg355@cornell.edu}. Supported by NSF Award CCF-1408673 and AFOSR Award FA9550-19-1-0183.}
\and
  Max Hopkins\thanks{Department of Computer Science and Engineering, UCSD, CA 92092. Email: \texttt{nmhopkin@eng.ucsd.edu}. Supported by NSF Award DGE-1650112.}
  \and
    Tali Kaufman\thanks{Department of Computer Science, Bar-Ilan University. Email: \texttt{kaufmant@mit.edu}. Supported by ERC and BSF.}
    \and
    Shachar Lovett\thanks{Department of Computer Science and Engineering, UCSD, CA 92092. Email: \texttt{slovett@cs.ucsd.edu}. Supported by NSF Award CCF-1909634.}
    \and
    Ruizhe Zhang\thanks{Department of Computer Science, UT Austin. Email: \texttt{ruizhe@utexas.edu}.}
}

\begin{document}

\maketitle

\begin{abstract}

We study the relationship between the underlying structure of posets and the spectral and combinatorial properties of their higher-order random walks. While fast mixing of random walks on hypergraphs (simplicial complexes) has led to myriad breakthroughs throughout theoretical computer science in the last five years, many other important applications (e.g.\ to locally testable codes, 2-2 games) rely on the more general non-simplicial structure of posets. These works make it clear that the global expansion properties of posets depend strongly on their underlying architecture (e.g.\ simplicial, cubical, linear algebraic), but the overall phenomenon remains poorly understood. In this work, we quantify the advantage of different poset architectures, highlighting how  structural \emph{regularity} controls the spectral decay and edge-expansion of corresponding random walks.

In particular, we show that the spectra of walks on expanding posets (Dikstein, Dinur, Filmus, Harsha RANDOM 2018) concentrate in strips around a small number of approximate eigenvalues controlled by the regularity of the underlying poset. This gives a simple condition to identify poset architectures (e.g. the Grassmann) that exhibit fast (exponential rate) decay of eigenvalues, versus architectures like hypergraphs whose eigenvalues decay slowly (linear rate)---a crucial distinction in applications to hardness of approximation and agreement testing such as the recent proof of the 2-2 Games Conjecture (Khot, Minzer, Safra FOCS 2018). We show these results lead to a tight variance-based characterization of edge-expansion on expanding posets generalizing (Bafna, Hopkins, Kaufman, and Lovett (SODA 2022)), and pay special attention to the case of the Grassmann where we show our results are tight for a natural set of sparsifications of the Grassmann graphs. We note for clarity that our results do not recover the characterization used in the proof of the 2-2 Games Conjecture which relies on $\ell_\infty$ rather than $\ell_2$-structure.

\end{abstract}
\section{Introduction}\label{sec:intro-intro}
Random walks on high dimensional expanders (HDX) have been the object of intense study in theoretical computer science in recent years. Starting with their original formulation by Kaufman and Mass \cite{kaufman2016high}, a series of works on the spectral structure of these walks \cite{kaufman2020high,dikstein2018boolean,alev2020improved} led to significant breakthroughs in approximate sampling \cite{anari2019log,alev2020improved,anari2020spectral,chen2020rapid,chen2021optimal,chen2021rapid,feng2021rapid,jain2021spectral,liu2021coupling,blanca2021mixing}, CSP-approximation \cite{alev2019approximating,bafna2020high}, error-correcting codes \cite{jeronimo2020unique,jeronimo2021near}, agreement testing \cite{dinur2017high,dikstein2019agreement,kaufman2020local}, and more. Most of these works focus on the structure of expansion in \textit{hypergraphs} (typically called \textit{simplicial complexes} in the HDX literature). On the other hand, it has become increasingly clear that hypergraphs are not always the right tool for the job---recent breakthroughs in locally testable \cite{dinur2021locally} and quantum LDPC codes \cite{panteleev2021asymptotically,lin2022c,leverrier2022quantum}, for instance, all rely crucially on \textit{cubical} structure not seen in hypergraphs, while many agreement testing results like the proof of the 2-2 Games Conjecture \cite{subhash2018pseudorandom} rely crucially on \textit{linear algebraic} rather than simplicial structure. 

In this work, we study a generalized notion of high dimensional expansion on \textit{partially ordered sets} (posets) introduced by Dikstein, Dinur, Filmus, and Harsha (DDFH) \cite{dikstein2018boolean} called \textit{expanding posets} (eposets). Random walks on eposets capture a broad range of important structures beyond their hypergraph analogs, including natural sparsifications of the Grassmann graphs that recently proved crucial to the resolution of the 2-2 Games Conjecture \cite{subhash2018pseudorandom, khot2017independent,dinur2018towards,dinur2018non,barak2018small,khot2018small}. DDFH's notion of eposets is a \emph{global} definition of high dimensional expansion based on a relaxation of Stanley's \cite{stanley1988differential} sequentially differential posets, a definition originally capturing both the Grassmanian and complete simplicial complex. 
More recently, Kaufman and Tessler (KT) \cite{kaufman2021local} have extended the study of eposets in two important aspects. First, in contrast to DDFH's original global definition, KT introduced the local-to-global study of high dimensional expansion in eposets. Second, they identified \textit{regularity} as a key parameter controlling expansion. In particular, the authors showed strengthened local-to-global theorems for strongly regular posets like the Grassmann, giving the first general formulation for characterizing expansion based on an eposet's underlying architecture.

While analysis of the second eigenvalue is certainly an important consideration (e.g.\ for mixing applications), a deeper understanding of the spectral structure of eposets is required for applications like the proof of the 2-2 Games Conjecture. As such, our main focus in this work lies in characterizing the spectral and combinatorial behavior of random walks on eposets \textit{beyond the second eigenvalue}. Strengthening DDFH and recent work of Bafna, Hopkins, Kaufman, and Lovett (BHKL) \cite{bafna2020high}, we prove that at a coarse level (walks on) eposets indeed exhibit the same spectral and combinatorial characteristics as expanding hypergraphs (e.g.\ spectral stripping, expansion of pseudorandom sets). On the other hand, as in KT, we show that the finer-grained properties of these objects are actually controlled by the underlying poset's regularity, including the \textit{rate of decay} of the spectrum and combinatorial expansion of associated random walks. This gives a stronger separation between structures like hypergraphs with weak (linear) eigenvalue decay, and Grassmann-based eposets with strong (exponential) eigenvalue decay (a crucial property in the proof of the 2-2 Games Conjecture \cite{subhash2018pseudorandom}).

In slightly more detail, we show that all eposets exhibit a behaviour called ``eigenstripping'' \cite{kaufman2020high,dikstein2018boolean,bafna2020high}: the spectrum of any associated random walk concentrates around a few unique approximate eigenvalues. Moreover, the approximate eigenvalues of walks on eposets are tightly controlled by the poset architecture's regularity\footnote{We will additionally assume a slightly stronger condition introduced in KT \cite{kaufman2021local} called \textit{middle regularity} throughout. See \Cref{sec:poset} for details.} $R(j,i)$, which denotes the total number of rank-$i$ elements\footnote{We consider \textit{regular graded posets}, where each element has a corresponding rank. In a hypergraph, for instance, rank is given by the size of a set, while in the Grassmann poset it is given by subspace dimension.} less than any fixed rank-$j$ element (see \Cref{sec:intro-background} for standard definitions). For simplicity, we specialize our result below to the popular ``lower'' or ``down-up'' walk (this simply corresponds to taking a random step down and back up the poset, again see \Cref{sec:intro-background}); a more involved version holds for higher order random walks in full generality. 
\begin{theorem}[Eigenstripping and Regularity (informal \Cref{cor:balanced-eig} and \Cref{claim:reg-lower})]
\label{intro-eigenstripping}
The spectrum of the lower walk $UD$ on a $k$-dimensional $\gamma$-eposet is concentrated in $(k+1)$ strips: 
\[
Spec(UD) \in \{1\} \cup \bigcup\limits_{i=1}^k [ \lambda_i(UD) + O_k(\gamma), \lambda_i(UD) - O_k(\gamma)],
\]
where the approximate eigenvalues $\lambda_i(UD)$ are determined by the poset's regularity:
\[
\lambda_i(UD) = \frac{R(k-1,i)}{R(k,i)}.
\]
\end{theorem}
\Cref{intro-eigenstripping} generalizes and tightens recent work on expanding hypergraphs of BHKL \cite[Theorem 2.2]{bafna2020high} (which itself extended a number of earlier works on the topic \cite{kaufman2020high,dikstein2018boolean,alev2019approximating}). 
%
Additionally, our result on the connection between regularity and approximate eigenvalues generalizes the work of KT \cite{kaufman2021local}, who show an analogous result for $\lambda_2$. \Cref{intro-eigenstripping} reveals a stark contrast between the spectral behavior of eposets with different regularity parameters. As a prototypical example, consider the case of hypergraphs versus subsets of the Grassmann ($k$-dimensional vector spaces over $\mathbb{F}_q^n$). In the former, each $k$-set contains ${k \choose i}$ $i$-sets, leading to approximate eigenvalues that decay \textit{linearly} ($\lambda_i \approx (k-i)/k$). On the other hand, each $k$-dimensional vector space contains ${k \choose i}_q$ $i$-dimensional subspaces, which leads to eigenvalues that decay \textit{exponentially} ($\lambda_i \approx q^{-i}$). The latter property, which we call \textit{strong decay} is often crucial in applications (e.g.\ for hardness of approximation \cite{subhash2018pseudorandom} or fast algorithms \cite{bafna2020high}), and while it is possible to recover strong decay on weaker posets by increasing the length of the walk \cite{bafna2020high}, this is often untenable in application due to the additional degrees of freedom it affords.\footnote{For instance such a walk might take exponential time to implement, or correspond to a more complicated agreement test than desired.}

The spectral structure of walks on eposets is closely related to their \textit{edge-expansion}, an important combinatorial property that has recently played a crucial role both in algorithms for \cite{bafna2020playing, bafna2020high} and hardness of unique games \cite{subhash2018pseudorandom}. The key insight in both cases lay in understanding \textit{the structure of non-expanding sets}. We give a tight understanding of this phenomenon across all eposets in the so-called $\ell_2$-regime \cite{bafna2020high}, where we show that expansion is tightly controlled by the behavior of local restrictions called \textit{links} (see \Cref{def:link}).
\begin{theorem}[Expansion in the $\ell_2$-Regime (informal \Cref{thm:hdx-expansion})]
The expansion of any $i$-link is almost exactly $1-\lambda_i(M)$. Conversely, any set with expansion less than $1-\lambda_{i+1}(M)$ has \textbf{high variance}
across $i$-links.
\end{theorem}
In \cite{bafna2020high}, it was shown this characterization allows for the application of a local-to-global algorithmic framework for unique games on such walks. This remains true on eposets,
and it is an interesting open question whether there are significant applications beyond those given in BHKL's original work.\footnote{While one can apply the framework to playing unique games on the Grassmann poset (or sparsifications thereof), the spectral parameters are such that this does not give a substantial improvement over standard algorithms \cite{arora2008unique}.} 

Finally, as an application of our structure theorems, we give an in-depth analysis of the $\ell_2$-structure of walks on expanding subsets of the Grassmann poset called $q$-eposets (first studied in \cite{dikstein2018boolean}). We focus in particular on the natural $q$-analog of an important set of walks called \textit{partial-swap walks} introduced by Alev, Jeronimo, and Tulsiani \cite{alev2019approximating} that generalize the Johnson graphs when applied to expanding hypergraphs. We show that applied to $q$-eposets, these objects give a natural set of walks generalizing the Grassmann graphs and further prove that our generic analysis for eposets gives a tight characterization of non-expansion in this setting. We note that this does not recover the result used for the proof of the 2-2 Games Conjecture which lies in the \textit{$\ell_\infty$-regime} (replacing variance above with maximum) and requires a \textit{dimension-independent} bound. This issue was recently (and independently) resolved for simplicial complexes in \cite{bafna2021hypercontractivity} and \cite{gur2021hypercontractivity}, and we view our work as an important step towards a more general understanding for families like the Grassmann beyond hypergraphs.

\subsection{Background}\label{sec:intro-background}
Before jumping into our results in any further formality, we'll briefly overview the theory of expanding posets and higher order random walks. All definitions are covered in full formality in \Cref{sec:prelims}. A $d$-dimensional graded poset is a set $X$ equipped with a partial order ``<'' and a ranking function $r: X \to [d]$ that respects the partial order and partitions $X$ into levels $X(0) \cup \ldots \cup X(d)$. When $x < y$ and $r(y)=r(x)+1$, we write $x \lessdot y$ or equivalently $y \gtrdot x$.\footnote{This is traditionally called a `covering relation.'} Finally, we will assume throughout this work that our posets are \textit{downward regular}: there exists a regularity function $R(k,i)$ such that every $k$-dimensional element is greater than exactly $R(k,i)$ $i$-dimensional elements.\footnote{For notational convenience, we also define $R(i,i)=1$ and $R(j,i)=0$ whenever $j<i$.}

Graded posets come equipped with a natural set of averaging operators called the \textit{up} and \textit{down} operators. Namely, for any function $f: X(i) \to \mathbb{R}$, these operators average $f$ up or down one level of the poset respectively:
\begin{align*}
    U_i f(x) &= \underset{y \lessdot x}{\mathbb{E}}[f(y)],\\
    D_if(y) &= \underset{x \gtrdot y}{\mathbb{E}}[f(x)].
\end{align*}
Composing the averaging operators leads to a natural notion of random walks on the underlying poset called \textit{higher order random walks} (HD-walks). The simplest example of such a walk is the \textit{upper walk} $D_{i+1}U_i$ which moves between elements $x,x' \in X(i)$ via a common element $y \in X(i+1)$ with $y> x,x'$. Similarly, the \textit{lower walk} $U_{i-1}D_i$ walks between $x,x' \in X(i)$ via a common $y \in X(i-1)$ with $y< x,x'$. It will also be useful at points to consider longer variants of the upper and lower walks called \textit{canonical walks} $\widehat{N}_k^i = D_{k+1} \circ \ldots \circ D_{k+i} \circ U_{k+i-1}\circ \ldots \circ U_k$ and $\widecheck{N}_k^i = U_k \circ \ldots \circ U_{k-i} \circ D_{k-i+1} \circ \ldots \circ D_k$ which similarly walk between $k$-dimensional elements in $X(k)$ via a shared element in $X(k+i)$ or $X(k-i)$ respectively.

Following DDFH \cite{dikstein2018boolean}, we call a poset a $(\delta,\gamma)$-expander for $\delta \in [0,1]^{d-1}$ and $\gamma \in \mathbb{R}_+$ if the upper and lower walks are spectrally similar up to a laziness factor:
\[
\norm{D_{i+1}U_{i} - (1-\delta_i)I - \delta_iU_{i-1}D_{i}} \leq \gamma.
\]
This generalizes standard spectral expansion which can be equivalently defined as looking at the spectral norm of $A_G - U_0D_1$, where $A_G$ (the adjacency matrix) is exactly the non-lazy upper walk. We note that under reasonable regularity conditions (see \cite{kaufman2021local,dikstein2018boolean}), this definition is equivalent to \textit{local-spectral expansion} \cite{dinur2017high}, which requires every local restriction of the poset to be an expander graph. While most of our results hold more generally, it will also be useful to assume a weak non-laziness condition on our underlying posets throughout that holds in most cases of interest (see \Cref{def:non-lazy}).
\subsection{Results}
With these definitions in mind, we can now cover our results in somewhat more formality. We split this section into three parts for readability: spectral stripping, characterizing edge expansion, and applications to the Grassmann.
\subsubsection{Eigenstripping}
We start with our generalized spectral stripping theorem for walks on expanding posets.
\begin{theorem}[Spectrum of HD-Walks (informal \Cref{cor:balanced-eig})]\label{intro:eposet-spectra}
Let $M$ be an HD-walk on the $k$th level of a $(\delta,\gamma)$-eposet. Then the spectrum of $M$ is highly concentrated in $k+1$ strips:
\[
\text{Spec}(M) \in \{1\} \cup \bigcup_{i=1}^k \left[\lambda_i(M) - e, \lambda_i(M) + e \right]
\]
where $e \leq O_{k,\delta}(\gamma)$. Moreover, the span of eigenvectors in the $i$th strip approximately correspond to functions lifted from $X(i)$ to $X(k)$.
\end{theorem}
This generalizes and improves an analogous result of BHKL \cite{bafna2020high} on expanding hypergraphs, which had sub-optimal error dependence of $O_k(\gamma^{1/2})$. The main improvement stems from an optimal spectral stripping result for arbitrary inner product spaces of independent interest.
\begin{theorem}[Eigenstripping (informal \Cref{thm:approx-ortho})]\label{intro:approx-ortho}
Let $M$ be a self-adjoint operator over an inner product space $V$, and $V=V^1 \oplus \ldots \oplus V^k$ be an ``approximate eigendecomposition'' in the sense that there exist $\{\lambda_i\}_{i=1}^k$ and sufficiently small error factors $\{c_i\}_{i=1}^k$ such that for all $f_i \in V^i$:
\[
\norm{Mf_i - \lambda_if_i}_2 \leq c_i\norm{f_i}.
\]
Then the spectrum of $M$ is concentrated around each $\lambda_i$:
\[
\text{Spec}(M) \subseteq \bigcup_{i=1}^k \left [ \lambda_i - c_i, \lambda_i + c_i\right ].
\]
\end{theorem}
Note that this result is tight---when there is $c_i$ ``error'' in our basis we cannot expect to have better than $c_i$ error in the resulting spectral strips. \Cref{intro:approx-ortho} improves over a preliminary result to this effect in \cite{bafna2020high} which had substantially worse dependence on $c_i$ and required much stronger assumptions.\footnote{It is also worth noting that the proof in this work is substantially simplified from \cite{bafna2020high}, requiring no linear algebraic manipulations at all.} \Cref{intro:eposet-spectra} then follows by work of DDFH (\cite[Theorem 8.6]{dikstein2018boolean}), who introduced a natural approximate eigendecomposition on eposets we call the HD-Level-Set Decomposition.

In full generality, the approximate eigenvalues in \Cref{intro:eposet-spectra} depend on the eposet parameters $\delta$, and can be fairly difficult to interpret. However, we show that under weak assumptions (see \Cref{sec:prelims}) the eigenvalues can be associated with the regularity of the underlying poset. We focus on the lower walks for simplicity, though the result can be similarly extended to general walks on eposets.
\begin{theorem}[Regularity Controls Spectral Decay (informal \Cref{claim:reg-lower})]\label{intro:regularity-thm}
The approximate eigenvalues of the lower walk $\widecheck{N}_k^{k-i}$ on a $(\delta,\gamma)$-eposet are controlled by the poset's regularity function:
\[
\lambda_j(\widecheck{N}_k^{k-i}) \in \frac{R(i,j)}{R(k,j)} \pm O_{k,\delta}(\gamma).
\]
\end{theorem}
As discussed in \Cref{sec:intro-intro}, this generalizes work of Kaufman and Tessler \cite{kaufman2021local} for the second eigenvalue of the upper/lower walks, and reveals a major distinction among poset architectures: posets with higher regularity enjoy faster decay of eigenvalues. 
We note that \Cref{intro-eigenstripping} can also be obtained by combining \Cref{intro:approx-ortho} with recent independent work of Dikstein, Dinur, Filmus, and Harsha on connections between eposets and regularity (namely in the recent update of their seminal eposet paper, see \cite[Section 8.4.1]{dikstein2018boolean}).

On a more concrete note, \Cref{intro:regularity-thm} gives a new method of identifying potential poset architectures exhibiting strong spectral decay in the sense that for any $\delta>0$, the lower walk only contains $O_\delta(1)$ approximate eigenvalues larger than $\delta$ (rather than a number growing with dimension). This property, referred to as \textit{constant ST-Rank} in the context of hypergraphs in \cite{bafna2020high}, is an important factor not only for the run-time of approximation algorithms on HDX \cite{bafna2020high}, but also for the soundness of the Grassmann-based agreement test in the proof of the 2-2 Games Conjecture \cite{subhash2018pseudorandom}.

\subsubsection{Characterizing Edge Expansion}
Much of our motivation for studying the spectrum of HD-walks comes from the desire to understand a fundamental combinatorial quantity of graphs called \textit{edge expansion}.
\begin{definition}[Edge Expansion]
Let $X$ be a graded poset and $M$ an HD-Walk on $X(k)$. The edge expansion of a subset $S \subset X(k)$ with respect to $M$ is
\[
\Phi(S) = \underset{v \sim S}{\mathbb{E}}\left[ M(v,X(k) \setminus S)\right],
\]
where 
\[
M(v, X(k) \setminus S) = \sum\limits_{y \in X(k) \setminus S} M(v,y)
\]
and $M(v,y)$ denotes the transition probability from $v$ to $y$.
\end{definition}
As mentioned in the introduction, characterizing the edge-expansion of sets in HD-walks has recently proven crucial to understanding both algorithms for \cite{bafna2020playing,bafna2020high} and hardness of unique games \cite{subhash2018pseudorandom}. On expanding hypergraphs, it has long been known that \textit{links} give the canonical example of small non-expanding sets.
\begin{definition}[Link]\label{def:link}
Let $X$ be a $d$-dimensional graded poset. The $k$-dimensional link of an element $\sigma \in X$ is the set of rank $k$ elements greater than $\sigma$:\footnote{We note that in the literature a link is usually defined to be all such elements, not just those of rank $k$. We adopt this notation since we are mostly interested in working at a fixed level of the complex.}
\[
X^k_\sigma = \{ y \in X(k): y > \sigma\}.
\]
We call the link of a rank-$i$ element an ``$i$-link.'' When the level $k$ is clear from context, we write $X_\sigma$ for $X^k_\sigma$ for simplicity.
\end{definition}
In greater detail, BHKL \cite{bafna2020high} proved that on hypergraphs, the expansion of links is exactly controlled by their corresponding spectral strip. While their proof of this fact relied crucially on simplicial structure, we show via a more general analysis that the result can be recovered for eposets.
\begin{theorem}[Expansion of Links (informal \Cref{thm:link-lower})]\label{thm:intro-link-lower}
Let $X$ be a $(\delta,\gamma)$-eposet and $M$ an HD-walk on $X(k)$.
Then for all $0 \leq i \leq k$ and $\tau \in X(i)$:
\[
\Phi(X_\tau) = 1 - \lambda_{i}(M) \pm O_{M,k,\delta}(\gamma).
\]
\end{theorem}
As an immediate consequence, we get that when $M$ is not a small-set expander, links are examples of small non-expanding sets. One might reasonably wonder whether the converse is true as well: \textit{are all non-expanding sets explained by links}? This requires a bit of formalization. Following BHKL's exposition \cite{bafna2020high}, given a set $S$ consider the function $L_{S,i}: X(i) \to \R$ that encodes the behavior of $S \subset X(k)$ on links:
\[
\forall \tau \in X(i): L_{S,i}(\tau) = \underset{X_\tau}{\mathbb{E}}[\id{S}] - \mathbb{E}[\id{S}].
\]
The statement ``Non-expansion is explained by links'' can then be interpreted as saying that a non-expanding set $S$ should be detectable by some simple measure of $L_{S,i}$. There are two standard formalizations of this idea studied in the literature: the $\ell_2$-regime, and the $\ell_\infty$-regime. These are captured by the following notion of \textit{pseudorandomness} based on $L_{S,i}$.
\begin{definition}[Pseudorandom Sets \cite{bafna2020high} (informal Definitions \ref{def:pseudorandom}, \ref{def:pseudorandom-infty})]
We say a set $S$ is $(\varepsilon,\ell)$-$\ell_2$-pseudorandom if\footnote{Throughout, $\norm{\cdot}_2$ will always refer to the \textit{expectation} norm $\norm{f}_{2}=\mathbb{E}[f^2]^{1/2}$.}
\[
\forall i \leq \ell: \norm{L_{S,i}}_2^2 \leq \varepsilon\mathbb{E}[\mathbbm{1}_S].
\]
A set is $(\varepsilon,\ell)$-$\ell_\infty$-pseudorandom if:
\[
\forall i \leq \ell: \norm{L_{S,i}}_\infty \leq \varepsilon.
\]
In cases that $\ell_2$ and $\ell_\infty$-pseudorandomness can be used interchangeably, we will simply write $(\varepsilon,\ell)$-pseudorandom.
\end{definition}

We prove that pseudorandom sets expand near-optimally.
\begin{theorem}[Pseudorandom Sets Expand (informal \Cref{thm:hdx-expansion})]\label{intro:PR-sets-expand}
Let $X$ be a $(\delta,\gamma)$-eposet and $M$ a walk on $X(k)$. Then the expansion of any $(\varepsilon,i)$-pseudorandom set $S$ is at least:
\[
\Phi(S) \geq 1 - \lambda_{i+1} - O_{\delta}(R(k,i)\varepsilon) - O_{k,\delta,M}(\gamma).
\]
\end{theorem}
In other words, any set with expansion less than $1-\lambda_{i+1}$ must have appreciable variance across links at level $i$. We note that the formal version of this result is essentially tight in the $\ell_2$-regime, but can be improved in many important cases in the $\ell_\infty$-regime. We'll discuss this further in the next section, especially in the context of the Grassmann poset.

Before this, however, it is worth separately mentioning the main technical component behind \Cref{intro:PR-sets-expand}, a result traditionally called a ``level-$i$'' inequality.
\begin{theorem}[Level-$i$ inequality (informal \Cref{thm:body-local-spec-proj})]\label{intro:body-local-spec-proj}
Let $X$ be a $(\delta,\gamma)$-eposet and $S \subset X(k)$ a $(\varepsilon, \ell)$-pseudorandom set. Then for all $1 \leq i \leq \ell$:
\[
|\langle \id{S}, \id{S,i} \rangle| \leq  \left(R(k,i)\varepsilon + O_{k,\delta}(\gamma)\right)\langle \id{S}, \id{S} \rangle
\]
where $\id{S,i}$ is the projection of $\id{S}$ onto the $i$th eigenstrip.\footnote{Note that since walks on eposets are simultaneously diagonalizable, the decomposition of $X$ into eigenstrips is independent of the choice of walk.}
\end{theorem}
In other words, pseudorandomness controls the projection of $S$ onto eigenstrips. Theorems \ref{intro:PR-sets-expand} and \ref{intro:body-local-spec-proj} recover the analogous optimal bounds for simplicial high dimensional expanders in \cite{bafna2020high}, where the regularity parameter $R(k,i)={k \choose i}$, and are tight in a number of other important settings such as the Grassmann (discussed below). \Cref{intro:PR-sets-expand} and \Cref{intro:body-local-spec-proj} can also be viewed as another separation between eposet architectures, this time in terms of \textit{combinatorial} rather than \textit{spectral} properties.

\subsubsection{Application: $q$-eposets and the Grassmann Graphs}

Finally, we'll discuss the application of our results to a particularly important class of eposets called ``$q$-eposets.'' Just like standard high dimensional expanders arise from expanding subsets of the complete complex (hypergraph), $q$-eposets arise from expanding subsets of the Grassmann Poset.
\begin{definition}[Grassmann Poset]
The Grassmann Poset is a graded poset $(X,<)$ where $X$ is the set of all subspaces of $\mathbb{F}_q^n$ of dimension at most $d$, the partial ordering ``$<$'' is given by inclusion, and the rank function is given by dimension.
\end{definition}
We call a (downward-closed) subset of the Grassmann poset a $q$-simplicial complex, and an expanding $q$-simplicial complex a $q$-eposet (see \Cref{sec:prelim-q-HDX} for exact details). Using our machinery for general eposets, we prove a tight level-$i$ inequality for pseudorandom sets.
\begin{corollary}[Grassmann level-$i$ inequality (informal \Cref{thm:grassmann-level})]\label{intro:grassmann-level}
Let $X$ be a $\gamma$-$q$-eposet and $S \subseteq X(k)$. If $S$ is $(\varepsilon,\ell)$-pseudorandom, then for all $1 \leq i \leq \ell$: 
\[
|\langle \id{S}, \id{S,i} \rangle| \leq  \left({k \choose i}_q\varepsilon + O_{q,k}(\gamma) \right) \langle \id{S},\id{S} \rangle
\]
where $\binom{k}{i}_q=\frac{(1-q^k)\cdots (1-q^{k-i+1})}{(1-q^i)\cdots (1-q)}$ is the Gaussian binomial coefficient.
\end{corollary}

\Cref{intro:grassmann-level} is tight in a few senses. First, we prove the bound cannot be improved by any constant factor, even in the $\ell_\infty$-regime. In other words, for every $c<1$, it is always possible to find an $(\varepsilon,i)$-pseudorandom function satisfying:
\[
|\langle \id{S}, \id{S,i} \rangle| > c\left({k \choose i}_q\varepsilon + O_{q,k}(\gamma) \right) \langle \id{S}, \id{S} \rangle.
\]
Furthermore, it is well known the dependence on $k$ in this result is necessary \cite{khot2017independent}, even if one is willing to suffer a worse dependence on the pseudorandomness $\varepsilon$. This is different from the case of standard simplicial complexes, where the dependence can be removed in the $\ell_\infty$-regime \cite{khot2018small,bafna2021hypercontractivity,gur2021hypercontractivity}. However, there is a crucial subtlety here. It is likely that the $k$-dependence in this result can be removed by \textit{changing the definition of pseudorandomness}. On the Grassmann poset itself, for instance, it is known that this can be done by replacing links with a closely related but finer-grained local structure known as ``zoom-in zoom-outs'' \cite{subhash2018pseudorandom}. Indeed, more generally it is an interesting open problem whether there always exists a notion of locality based on the underlying poset structure that gives rise to $k$-independent bounds in the $\ell_\infty$-regime.

We close out the section by looking at an application of this level-$i$ inequality to studying edge-expansion in an important class of walks that give rise to the well-studied \textit{Grassmann graphs}.
\begin{definition}[Grassmann Graphs]
The Grassmann Graph $J_q(n,k,t)$ is the graph on $k$-dimensional subspaces of $\mathbb{F}_q^n$ where $(V,W) \in E$ exactly when $\text{dim}(V \cap W)=t$.
\end{definition}
It is easy to see that the non-lazy upper walk on the Grassmann poset is exactly the Grassmann graph $J_q(n,k,k-1)$. In fact, it is possible to express any $J_q(n,k,t)$ as a sum of standard higher order random walks.
\begin{proposition}[Grassmann Graphs are HD-Walks (informal \Cref{prop:Grassmann-decomp})]\label{intro:prop-grass-graphs}
The Grassmann graphs are a hypergeometric sum of canonical walks:
\[
 J_q(n,k,t) =\frac{1}{q^{(k-t)^2}{k \choose t}_q}\sum\limits_{i=0}^{k-t} (-1)^{k-t-i}q^{{k-t-i \choose 2}}{k-t \choose i}_q{k+i \choose i}_q
 N_k^{i}.
 \]
\end{proposition}
In \Cref{sec:q-HDX} we prove a more general version of this result for any $q$-simplicial complex. This leads to a set of natural sparsifications of the Grassmann graphs that may be of independent interest for agreement testing, PCPs, and hardness of approximation. For simplicity, on a given $q$-simplicial complex $X$, we'll refer to these ``sparsified'' Grassmann graphs as $J_{X,q}(n,k,t)$ for the moment (more formally they are the ``partial-swap walks,'' see \Cref{sec:prelim-q-HDX}). With this in mind, let's take a look at what our level-$i$ inequality implies for the edge-expansion of these graphs.
\begin{corollary}[$q$-eposets Edge-Expansion (informal \Cref{cor:grassmann-exp})]\label{cor:intro-grassmann-exp}
Let $X$ be a $d$-dimensional $\gamma$-$q$-eposet and $S \subset X(k)$ a $(\varepsilon,\ell)$-pseudorandom set. Then the expansion of $S$ with respect to the sparsified Grassmann graph $J_{X,q}(n,k,t)$ is at least:
\[
\Phi(S) \geq 1 - \mathbb{E}[\id{S}] - \varepsilon \sum\limits_{i=1}^\ell {t \choose i}_q  - q^{-(\ell+1)j} - O_{q,k}(\gamma).
\]
\end{corollary}
In practice, $t$ is generally thought of as being $\Omega(k)$ (or even $k-O(1)$), which results in a $k$-dependent bound. It remains an open problem whether a $k$-independent version can be proved for any $q$-eposet beyond the Grassmann poset itself. We conjecture such a result should indeed hold (albeit under a different notion of pseudorandomness), and may follow from $q$-analog analysis of recent work proving $k$-independent bounds for standard expanding hypergraphs \cite{bafna2021hypercontractivity,gur2021hypercontractivity}.
\subsection{Related Work}
\paragraph{Higher Order Random Walks.} Higher order random walks were first introduced in 2016 by Kaufman and Mass \cite{kaufman2016high}. Their spectral structure was later elucidated in a series of works by Kaufman and Oppenheim \cite{kaufman2020high}, DDFH \cite{dikstein2018boolean}, Alev, Jeronimo, and Tulsiani \cite{alev2019approximating}, Alev and Lau \cite{alev2020improved}, and finally BHKL \cite{bafna2020high}. With the exception of DDFH (who only worked with approximate eigenvectors without analyzing the true spectrum), all of these works focused on hypergraphs rather than general posets. Our spectral stripping theorem for eposets essentially follows from combining eposet machinery developed by DDFH with our improved variant of BHKL's general spectral stripping theorem. 

Higher order random walks have also seen an impressive number of applications in recent years, frequently closely tied to analysis of their spectral structure. This has included breakthrough works on approximate sampling \cite{anari2019log,alev2020improved,anari2020spectral,chen2020rapid,chen2021optimal,chen2021rapid,feng2021rapid,jain2021spectral,liu2021coupling,blanca2021mixing}, CSP-approximation \cite{alev2019approximating,bafna2020high}, error-correcting codes \cite{jeronimo2020unique,jeronimo2021near}, and agreement testing \cite{dinur2017high,dikstein2019agreement,kaufman2020local}. In this vein, our work is most closely related to that of Bafna, Barak, Kothari, Schramm, and Steurer \cite{bafna2020playing}, and BHKL \cite{bafna2020high}, who used the spectral and combinatorial structure of HD-walks to build new algorithms for unique games. As previously discussed, the generalized analysis in this paper also lends itself to the algorithmic techniques developed in those works, but we do not know of any interesting examples beyond those covered in BHKL.

\paragraph{High Dimensional Expansion Beyond Hypergraphs.} Most works listed above (and indeed in the high dimensional expansion literature in general) focus only on the setting of hypergraphs. However, recent years have also seen the nascent development and application of expansion beyond this setting \cite{dinur2021,panteleev2021asymptotically,lin2022c,leverrier2022quantum,Hopkins2022ExplicitLB}, including the seminal work of DDFH \cite{dikstein2018boolean} on expanding posets as well as more recent breakthroughs on locally testable and quantum codes \cite{dinur2021locally,panteleev2021asymptotically}. While DDFH largely viewed eposets as having similar structure (with the exception of the Grassmann), we strengthen the case that different underlying poset architectures exhibit different properties. This complements the recent result of Kaufman and Tessler \cite{kaufman2021local}, who showed that expanding posets with strong regularity conditions such as the Grassmann exhibit more favorable properties with respect to the second eigenvalue. Our results provide a statement of the same flavor looking at the entire spectrum, along with additional separations in more combinatorial settings. We note that a related connection between poset regularity and the approximate spectrum of walks on eposets was independently developed by DDFH in a recent update of their seminal work \cite{dikstein2018boolean}.

\paragraph{Expansion and Unique Games.}
One of the major motivations behind this work is towards building a more general framework for understanding the structure underlying the Unique Games Conjecture \cite{khot2002power}, a major open problem in complexity theory that implies optimal hardness of approximation results for a large swath of combinatorial optimization problems (see e.g.\ Khot's survey \cite{khot2005unique}). In 2018, Khot, Minzer, and Safra \cite{subhash2018pseudorandom} made a major breakthrough towards the UGC in proving a weaker variant known as the 2-2 Games Conjecture, completing a long line of work in this direction \cite{khot2017independent,dinur2018towards,dinur2018non,barak2018small,khot2018small,subhash2018pseudorandom}. The key to the proof lay in a result known as the ``Grassmann expansion hypothesis,'' which stated that any non-expanding set in the Grassmann graph $J_q(d,k,k-1)$ had to be non-trivially concentrated inside a local-structure called ``zoom-in zoom-outs.'' As noted in the previous section, this result differs from our analysis in two key ways: it lies in the $\ell_\infty$-regime, and must be totally independent of dimension.

Unfortunately, very little progress has been made towards the UGC since this result. This is in part because KMS' proof of the Grassmann expansion hypothesis, while a tour de force, is complicated and highly tailored to the exact structure of the Grassmann. To our knowledge, the same proof cannot be used, for instance, to resolve the related ``shortcode expansion hypotheses'' beyond degree-2, similar conjectures offered by Barak, Kothari, and Steurer \cite{barak2018small} in an effort to push beyond hardness of 2-2 Games. Just as the $\ell_2$-regime analysis of DDFH and BHKL recently lead to a dimension independent bound in the $\ell_\infty$-regime for standard HDX \cite{bafna2021hypercontractivity,gur2021hypercontractivity}, we expect the groundwork laid in this paper will be important for proving generalized dimension independent expansion hypotheses in the $\ell_\infty$-regime beyond the special case of the Grassmann graphs.

\section{Preliminaries}\label{sec:prelims}
Before jumping into the details in full formality, we give a more careful review of background definitions regarding expanding posets, higher order random walks, and the Grassmann.
\subsection{Graded Posets}\label{sec:poset}
We start with eposets' underlying structure, graded posets. A partially ordered set (poset) $P=(X,<)$ is a set of elements $X$ endowed with a partial order ``$<$''. A graded poset comes equipped additionally with a rank function $r: X \to \mathbb{N}$ satisfying two properties:
\begin{enumerate}
    \item $r$ preserves ``$<$'': if $y < x$, then $r(y) < r(x)$.
    \item $r$ preserves cover relations: if $x$ is the smallest element greater than $y$, then $r(x)=r(y)+1$.
\end{enumerate}
In other words, the function $r$ partitions $X$ into subsets by rank:
\[
X(0) \cup \ldots \cup X(d),
\]
where $\max_{X}(r)=d$, and $X(i)=r^{-1}(i)$. We refer to a poset with maximum rank $d$ as ``$d$-dimensional'', and elements in $X(i)$ as ``$i$-faces''. Throughout this work, we will consider only $d$-dimensional graded posets with two additional restrictions:
\begin{enumerate}
    \item They have a unique minimal element, i.e. $|X(0)|=1$.
    \item They are ``pure'': all maximal elements have rank $d$.
\end{enumerate}
Finally, many graded posets of interest satisfy certain regularity conditions which will be crucial to our analysis. The first condition of interest is a natural notion called \textit{downward} regularity.
\begin{definition}[Downward Regularity]
We call a $d$-dimensional graded poset downward regular if for all $i \leq d$ there exists some constant $R(i)$ such that every element $x \in X(i)$ covers exactly $R(i)$ elements $y \in X(i-1)$.
\end{definition}
Second, we will also need a useful notion called \textit{middle regularity} that ensures uniformity across multiple levels of the poset.
\begin{definition}[Middle Regularity]
We call a $d$-dimensional graded poset middle-regular if for all $0 \leq i \leq k \leq d$, there exists a constant $m(k,i)$ such that for any $x_k \in X(k)$ and $x_i \in X(i)$ satisfying $x_k>x_i$, there are exactly $m(k,i)$ chains\footnote{Such objects are sometimes called flags, e.g.\ in the case of the Grassmann poset.} of elements $x_k > x_{k-1} > \ldots > x_{i+1} > x_i$ where each $x_j \in X(j)$.
\end{definition}
We call a poset regular if it is both downward and middle regular. We note that regular posets also have the nice property that for any dimensions $i < k$, there exists a higher order regularity function $R(k,i)$ such that any $x \in X(k)$ is greater than exactly $R(k,i)$ elements in $X(i)$ (see \Cref{app:regularity}). We will use this notation throughout. For notational convenience, we also define $R(i,i)=1$ and $R(j,i)=0$ whenever $j<i$.

Important examples of regular posets include pure simplicial complexes and the Grassmann poset (subspaces of $\mathbb{F}^n_q$ ordered by inclusion). We will assume all posets we discuss in this work are regular from this point forward.
\subsection{Measured Posets and The Random Walk Operators}
Higher order random walks may be defined over posets in a very similar fashion to simplicial complexes. The main difference is simply that ``inclusion'' is replaced with the poset order relation. Just as we might want these walks on HDX to have non-uniform weights, the same is true for posets, which can be analogously endowed with a distribution over levels. In slightly more detail, a \textit{measured poset} is a graded poset $X$ endowed with a distribution 
$\Pi=(\pi_0,\ldots,\pi_d)$, where each marginal $\pi_i$ is a distribution over $X(i)$. While measured posets may be defined in further generality (cf.\ \cite[Definition 8.1]{dikstein2018boolean}), we will focus on the case in which the distribution $\Pi$ is induced entirely from $\pi_d$, analogous to weighted simplicial complexes. More formally, we have that for every $0 \leq i < d$:
\[
\pi_i(x) = \frac{1}{R(i+1,i)}\sum\limits_{y \gtrdot x}\pi_{i+1}(y).
\]
In other words, each lower dimensional distribution $\pi_i$ may be induced through the following process: an element $y \in X(i+1)$ is selected with respect to $\pi_{i+1}$, and an element $x \in X(i)$ such that $x<y$ is then chosen uniformly at random.

The averaging operators $U$ and $D$ are defined analogously to their notions on simplicial complexes, with the main change being the use of the general regularity function $R(i+1,i)$:
\begin{align*}
U_if(y) = &~ \frac{1}{R(i+1,i)}\sum\limits_{x \lessdot y}f(x),\\
D_{i+1}f(x) = &~ \frac{1}{\pi_{i+1}(X_{x})}\sum\limits_{y \gtrdot x}\pi_{i+1}(y)f(y),
\end{align*}
where for $i<k$ and $x \in X(i)$,
\[
\pi_{k}(X_x) = \sum\limits_{y \in X(k): y > x} \pi_{k}(y) = R(k,i)\pi_i(x)
\]
is the appropriate normalization factor (we will use this notation throughout). On regular posets, it is useful to note that the up operators compose nicely, and in particular that:
\[
U^k_i f(y) \coloneqq U_{k-1} \circ \ldots \circ U_i f(y) = \frac{1}{R(k,i)}\sum\limits_{x \in X(i): x < y}f(x)
\]
(see \Cref{app:regularity}). Furthermore, just like on simplicial complexes, the down and up operators are adjoint with respect to the standard inner product on measured posets:
\[
\langle f,g \rangle_{X(k)} = \sum\limits_{\tau \in X(k)} \pi_k(\tau)f(\tau)g(\tau),
\]
that is to say for any $f: X(k) \to \R$ and $g: X(k-1) \to \R$:
\[
\langle f,U_{k-1}g\rangle_{X(k)}=\langle D_k f,g\rangle_{X(k-1)}.
\]
Note that we'll generally drop the $X(k)$ from the notation when clear from context. This useful fact allows us to define basic self-adjoint notions of higher order random walks just like on simplicial complexes. 

\subsection{Higher Order Random Walks}
Let $C_k$ denote the space of functions $f: X(k) \to \mathbb{R}$. We define a natural set of random walk operators via the averaging operators.
\begin{definition}[$k$-Dimensional Pure Walk \cite{kaufman2016high,dikstein2018boolean,alev2019approximating}]
Given a measured poset $(X,\Pi)$, a $k$-dimensional pure walk $Y: C_k \to C_k$ on $(X,\Pi)$ (of height $h(Y)$) is a composition:
\[
Y = Z_{2h(Y)} \circ \cdots \circ Z_{1},
\]
where each $Z_i$ is a copy of $D$ or $U$, and there are $h(Y)$ of each type.
\end{definition}
Following AJT and BHKL, we define general higher order random walks to be affine combinations\footnote{An affine combination is a linear combination whose coefficients sum to $1$.} of pure walks.
\begin{definition}[HD-walk]
Let $X$ be a graded poset. Let $\mathcal Y$ be a family of pure walks $Y: C_k \to C_k$ on $(X,\Pi)$. We call an affine combination 
\[
M = \sum\limits_{Y \in \mathcal Y} \alpha_Y Y
\]
a $k$-dimensional HD-walk on $(X,\Pi)$ if it is stochastic and self-adjoint. The height of $M$, denoted $h(M)$, is the maximum height of any pure $Y \in \mathcal Y$ with a non-zero coefficient. The weight of $M$, denoted $w(M)$, is $|\alpha|_1$.
\end{definition}
While most of our results will hold for general HD-walks (or at least some large subclass), we pay special attention to a basic class of pure walks that have seen the most study in the literature: \textit{canonical walks}.
\begin{definition}[Canonical Walk]
Given a $d$-dimensional measured poset $(X,\Pi)$ and parameters $k+j \leq d$, the upper canonical walk $\widehat{N}_k^j$ is:
\[
\widehat{N}_k^j = D^{k+j}_{k}U^{k+j}_{k},
\]
and for $j \leq k$ the lower canonical walk $\widecheck{N}^j_k$ is:
\[
\widecheck{N}^j_k = U_{k-j}^kD^k_{k-j},
\]
where $U^k_\ell = U_{k-1}\ldots U_{\ell}$, and $D^k_\ell = D_{\ell+1}\ldots D_{k}$.
\end{definition}
Since the non-zero spectrum of $\widehat{N}_k^j$ and $\widecheck{N}_{k+j}^j$ are equivalent (c.f.\ \cite{alev2020improved}), we focus in this work mostly on the upper walks which we write simply as $N_k^j$.

For certain specially structured posets, we will also study an important class of HD-walks known as (partial) \textit{swap walks}. We will introduce these well-studied walks in more detail in \Cref{sec:prelim-q-HDX}, and for now simply note that they give a direct generalization of the Johnson and Grassmann graphs when applied to the complete complex and Grassmann poset respectively.

\subsection{Expanding Posets and the HD-Level-Set Decomposition}
Dikstein, Dinur, Filmus, and Harsha \cite{dikstein2018boolean} observed that one can use the averaging operators to define a natural extension of spectral expansion to graded posets. Their definition is inspired by the fact that $\gamma$-spectral expansion on a standard graph $G$ can be restated as a bound on the spectral norm of the adjacency matrix minus its stationary operator:
\[
\norm{A_G - UD} \leq \gamma.
\]
Informally, DDFH's definition can be thought of as stating that this relation holds for every level of a higher dimensional poset.
\begin{definition}[eposet \cite{dikstein2018boolean}]\label{def:eposet}
Let $(X,\Pi)$ be a measured poset, $\delta \in [0,1]^{d-1}$, and $\gamma < 1$. $X$ is an $(\delta,\gamma)$-eposet if for all $1 \leq i \leq d-1$:
\[
\norm{D_{i+1}U_i - (1-\delta_i) I - \delta_i U_{i-1}D_i} \leq \gamma
\].
\end{definition}
We note that for a broad range of posets, this definition is actually equivalent (up to constants) to \textit{local-spectral expansion}, a popular notion of high dimensional expansion introduced by Dinur and Kaufman \cite{dinur2017high}. This was originally proved for simplicial complexes by DDFH \cite{dikstein2018boolean}, and later extended to a more general class of posets by Kaufman and Tessler \cite{kaufman2021local}. It is also worth noting that when $\gamma=0$, posets satisfying the guarantee in \Cref{def:eposet} are known as \textit{sequentially differential}, and were actually introduced much earlier by Stanley \cite{stanley1988differential} in the late 80s.

Much of our analysis in this work will be based off of an elegant approximate Fourier decomposition for eposets introduced by DDFH \cite{dikstein2018boolean}.
\begin{theorem}[HD-Level-Set Decomposition, Theorem 8.2 \cite{dikstein2018boolean}]\label{thm:decomp-ddfh}
Let $(X,\Pi)$ be a $d$-dimensional $(\delta,\gamma)$-eposet with $\gamma$ sufficiently small. For all $0 \leq k \leq d$, let
\[
H^0=C_0, H^i=\text{Ker}(D_i), V_k^i = U^{k}_iH^i.
\]
Then:
\[
C_k = V^0_k \oplus \ldots \oplus V^k_k.
\]
In other words, every $f \in C_k$ has a unique decomposition $f=f_0+\ldots+f_k$ such that $f_i=U^{k}_ig_i$ for $g_i \in \text{Ker}(D_i)$.
\end{theorem}
It is well known that the HD-Level-Set Decomposition is approximately an eigenbasis for HD-walks on simplicial complex \cite{dikstein2018boolean,alev2019approximating,bafna2020high}. We show this statement extends to all eposets in \Cref{sec:spectrum-eposet} (extending DDFH's similar analysis of the upper walk $N_k^1$).

Finally, before moving on, we will assume for simplicity throughout this work an additional property of eposets we called (approximate) non-laziness.
\begin{definition}[$\beta$-non-Lazy Eposets]\label{def:non-lazy}
Let $(X,\Pi)$ be a $d$-dimensional measured poset. We call $(X,\Pi)$ $\beta$-non-lazy if for all $1 \leq i \leq d$, the laziness of the lower walk satisfies:
\[
\max_{\sigma \in X(i)}\{\id{\sigma}^T U_{i-1}D_i \id{\sigma}\} \leq \beta.
\]
\end{definition}
Another way to think about this condition is that no element in the poset carries too much weight, even upon conditioning. All of our results hold for general eposets,\footnote{The one exception is the lower bound of \Cref{thm:intro-link-lower}.} but their form is significantly more interpretable when the poset is additionally non-lazy. In fact, most $\gamma$-eposets of interest are $O(\gamma)$-non-lazy. It is easy to see for instance that any ``$\gamma$-local-spectral'' expander satisfies this condition, an equivalent notion of expansion to $\gamma$-eposets under suitable regularity conditions \cite{kaufman2021local}. We discuss this further in \Cref{app:regularity}.

\subsection{The Grassmann Poset and $q$-eposets}\label{sec:prelim-q-HDX}
At the moment, there are only two known families of expanding posets of significant interest in the literature: those based on pure simplicial complexes (the downward closure of a $k$-uniform hypergraph), and pure $q$-simplicial complexes (the analogous notion over subspaces). The $\ell_2$-structure of the former set of objects is studied in detail in \cite{bafna2020high}. In this work, we will focus on the latter which has seen less attention in the literature, but is responsible for a number of important results including the resolution of the 2-to-2 Games Conjecture \cite{subhash2018pseudorandom}.

\begin{definition}[$q$-simplicial complex]
Let $G_q(n,d)$ denote the $d$-dimensional subspaces of $\mathbb{F}_q^n$. A weighted, pure $q$-simplicial complex $(X,\Pi)$ is given by a family of subspaces $X \subseteq G_q(n,d)$ and a distribution $\Pi$ over $X$. We will usually consider the downward closure of $X$ in the following sense:
\[
X = X(0) \cup \ldots \cup X(d),
\]
where $X(i) \subseteq G_q(n,i)$ consists of all $i$-dimensional subspaces contained in some element in $X=X(d)$. Further, on each level $X(i)$, $\Pi$ induces a natural distribution $\pi_i$:
\[
\forall V \in X(i): \pi_i(V) = \frac{1}{{d \choose i}_q}\sum\limits_{W \in X(d): W \supset V} \pi_d(W),
\]
where $\pi_d=\Pi$ and $\binom{d}{i}_q=\frac{(1-q^d)\cdots (1-q^{d-i+1})}{(1-q^i)\cdots (1-q)}$ is the Gaussian binomial coefficient.
\end{definition}
The most basic example of a $q$-simplicial complex is the Grassmann poset, which corresponds to taking $X=G_q(n,d)$. This is the $q$-analog of the complete simplicial complex. The Grassmann poset is well known to be a expander in this sense (see e.g.\ \cite{stanley1988differential})---in fact it is a sequentially differential poset with parameters
\[
\delta_i = \frac{(q^i-1)(q^{n-i+1}-1)}{(q^{i+1}-1)(q^{n-i}-1)},
\]
the $q$-analog of the eposet parameters for the complete complex \cite{dikstein2018boolean}. With this in mind, let's define a special class of eposets based on $q$-simplicial complexes.
\begin{definition}[$\gamma$-$q$-eposet \cite{dikstein2018boolean}]
A pure, $d$-dimensional weighted $q$-simplicial complex $(X,\Pi)$ is a $\gamma$-$q$-eposet if it is a $(\delta,\gamma)$-eposet satisfying $\delta_i = q\frac{q^i-1}{q^{i+1}-1}$ for all $1 \leq i \leq d-1$.
\end{definition}
Constructing bounded-degree $q$-eposets (a problem proposed by DDFH \cite{dikstein2018boolean}) remains an interesting open problem. Kaufman and Tessler \cite{kaufman2021local} recently made some progress in this direction, but the expansion parameter of their construction is fairly poor (around $1/2$).

Finally, in our applications to the Grassmann we'll focus our attention on a particularly important class of walks called \textit{partial-swap walks}. These should essentially be thought of as non-lazy variants of the upper canonical walks.
\begin{definition}[Partial-Swap Walk]
Let $(X,\Pi)$ be a weighted, $d$-dimensional $q$-simplicial complex. The partial-swap walk $S^j_k$ is the restriction of the canonical walk $N^j_k$ to faces whose intersection has dimension $k-j$. In other words, if $|V \cap W| > k-j$ then $S^j_k(V,W)=0$, and otherwise $S^j_k(V,W) \ \propto \ N^j_k(V,W)$.
\end{definition}
When applied to the Grassmann poset itself, it is clear by symmetry that the partial-swap walk $S_k^j$ returns exactly the Grassmann graph $J_q(d,k,k-j)$. On the other hand, it is not immediately obvious these objects are even HD-walks when applied to a generic $q$-simplicial complex. We prove this is the case in \Cref{sec:q-HDX}.

\section{Approximate Eigendecompositions and Eigenstripping}
With preliminaries out of the way, we can move on to understanding HD-walks' spectral structure. It turns out that on expanding posets, these walks exhibit almost exactly the same properties as on the special case of simplicial complexes studied in \cite{kaufman2020high,dikstein2018boolean,alev2019approximating,bafna2020high}: a walk's spectrum lies concentrated in strips corresponding to levels of the HD-Level-Set Decomposition. The key to proving this lies in a more general theorem characterizing the spectral structure of any inner product space admitting a ``approximate eigendecomposition.''
\begin{definition}[Approximate Eigendecomposition \cite{bafna2020high}]\label{def:approx-eigen}
Let $M$ be an operator over an inner product space $V$. A decomposition $V=V^1 \oplus \ldots \oplus V^k$ is called a $(\{\lambda_i\}_{i=1}^k,\{c_i\}_{i=1}^k)$-approximate eigendecomposition if for all $i$ and $v_i \in V^i$, $Mv_i$ is close to $\lambda_i v_i$:
\[
\norm{Mv_i - \lambda_i v_i} \leq c_i \norm{v_i}.
\]
We will always assume for simplicity (and without loss of generality) that the $\lambda_i$ are sorted: $\lambda_1 \geq \ldots \geq \lambda_k$.
\end{definition}
BHKL \cite{bafna2020high} proved that as long as the $c_i$ are sufficiently small, each $V^i$ (loosely) corresponds to an ``eigenstrip,'' the span of eigenvectors with eigenvalue closely concentrated around $\lambda_i$, and that these strips account of the entire spectrum of $M$. While sufficient for their purposes, their proof of this result was complicated and resulted in a variety of sub-optimal parameters. We give a tight variant of this result and significantly simplify the proof.

\begin{theorem}[Eigenstripping]\label{thm:approx-ortho}
Let $M$ be a self-adjoint operator over an inner product space $V$, and $V=V^1 \oplus \ldots \oplus V^k$ a $(\{\lambda_i\}_{i=1}^k,\{c_i\}_{i=1}^k)$-approximate eigendecomposition. Then as long as $c_{i}+c_{i+1} < \lambda_{i} - \lambda_{i+1}$, the spectrum of $M$ is concentrated around each $\lambda_i$:
\[
\text{Spec}(M) \subseteq \bigcup_{i=1}^k \left [ \lambda_i - c_i, \lambda_i + c_i\right ]
\]
\end{theorem}
\begin{proof}
The idea is to examine for each $i$ the operator $M_i^2 = (M- \lambda_i I)^2$. In particular, we claim it is enough to show the following:
\begin{claim}\label{claim:stripping}
For all $1 \leq i \leq k$, $\text{Spec}(M_i^2)$ contains $dim(V^i)$ eigenvalues less than $c_i^2$.
\end{claim}
Let's see why this implies the desired result. Notice that the eigenvalues of $M_i^2$ are exactly $(\mu-\lambda_i)^2$ for each $\mu$ in $Spec(M)$ (with matching multiplicities), and therefore that any eigenvalue $\mu_i \in Spec(M_i^2)$ less than $c_i^2$ implies the existence of a corresponding eigenvalue of $M$ in $[\lambda_i \pm c_i]$. If each $M_i^2$ has $\text{dim}(V^i)$ eigenvalues less than $c_i^2$, then $M$ has at least $\text{dim}(V^i)$ eigenvalues in each interval $[\lambda_i \pm c_i]$. Moreover, since these intervals are disjoint by assumption and $\sum \text{dim}(V^i) = \text{dim}(V)$, this must account for all eigenvalues of $M$.

It remains to prove the claim, which is essentially an immediate application of Courant-Fischer theorem \cite{fischer1905quadratische}. 
\begin{proof}[Proof of \Cref{claim:stripping}]
The Courant-Fischer theorem states that the $k$th smallest eigenvalue of a self-adjoint operator $A$ is:
\[
\lambda_{n-k+1} = \min_{U}\left\{\max_{f \in U}\left\{\frac{\langle f, Af \rangle}{\langle f,f \rangle}\right\} ~\middle|~ \text{dim}(U) = k\right\}.
\]
Setting $U=V^i$, $A=M_i^2$ and $k=\dim(V^i)$ gives the claim:
\[
\lambda_{n-k+1}(M_i^2) \leq \max_{f \in V^i}\left\{\frac{\langle f, M_i^2f \rangle}{\langle f,f \rangle}\right\} = \max_{f \in V^i}\left\{\frac{\norm{ (M- \lambda_i I)f}_2^2}{\langle f,f \rangle}\right\} \leq c_i^2
\]
since $(M-\lambda_i I)$ is self-adjoint and $\bigoplus_{i\in [k]} V^i$ is a $(\{\lambda_i\}_{i=1}^k,\{c_i\}_{i=1}^k)$-approximate eigendecomposition.
\end{proof}

\end{proof}
Note that this result is also trivially tight, as any true eigendecomposition is also a $( \{\lambda_i \pm c_i\}, \{c_i\})$-approximate eigendecomposition. We also note that similar strategies have been used in the numerical analysis literature (see e.g.\ \cite{horn1998eigenvalue}).

\section{Spectra of HD-walks}\label{sec:spectrum-eposet}
Given \Cref{thm:approx-ortho}, it is enough to prove that the HD-Level-Set Decomposition is an approximate eigenbasis for any HD-walk. This follows by the same inductive argument as for local-spectral expanders in \cite{bafna2020high}, where the only difference is that somewhat more care is required to deal with general eposet parameters. To start, it will be useful to lay out some notation along with a simple observation from repeated application of \Cref{def:eposet}.
\begin{lemma}[{\cite[Claim 8.8]{dikstein2018boolean}}]\label{lemma:DU-UD}
Let $(X,\Pi)$ be a $d$-dimensional $(\delta,\gamma)$-eposet. Then
\[
\norm{D_{k+1}U^{k+1}_{k-j} - (1-\delta^k_j)U^{k}_{k-j} - \delta_j^k U^{k}_{k-j-1}D_{k-j}} \leq \gamma^k_j,
\]
where
\[
\delta^k_{-1}=1,~\delta_j^k = \prod\limits_{i=k-j}^k \delta_i,
~\gamma^k_j = \gamma \sum\limits^{j-1}_{i=-1}\delta^k_i.
\]
\end{lemma}
Applying this fact inductively implies that functions in the HD-Level-Set Decomposition are close to being eigenvectors.
\begin{proposition}\label{prop:pure-eig-vals}
Let $(X,\Pi)$ be a $(\delta,\gamma)$-eposet, and $Y$ the pure balanced walk of height $j$, with down operators at positions $(i_1,\ldots,i_{j})$. For $1 \leq \ell \leq k$, let $f_\ell=U^{k}_\ell g_\ell$ for some $g_\ell \in H^\ell$, and let
\[
\delta_j^k = \prod\limits_{i=k-j}^k \delta_i,
~\gamma^k_j = \gamma \sum\limits^{j-1}_{i=-1}\delta^k_i,
\]
where $\delta_{i}^k = 1$ for any $i<0$ for notational convenience. Then $f_\ell$ is an approximate eigenvector of $Y$:
\[
\norm{Yf_\ell - \prod\limits_{s=1}^j \left (1 - \delta_{k-2s+i_s-\ell}^{k-2s+i_s} \right )f_\ell} \leq \norm{g_\ell}\sum\limits_{s=1}^j \gamma^{k-2s+i_s}_{k-2s+i_s-\ell} \prod\limits_{t=1}^{s-1}\left ( 1 - \delta_{k-2t+i_t-\ell}^{k-2t+i_t} \right ) \leq (j+k)j\gamma\norm{g_\ell}.
\]
\end{proposition}
\begin{proof}
We prove a slightly stronger statement to simplify the induction. For $b>0$, let $Y_{j}^b: C_\ell \to C_{\ell+b}$ denote an unbalanced walk with $j$ down operators, and $j+b$ up operators. If $Y_{j}^b$ has down operators in positions $(i_{1},\ldots,i_{j})$ and $g_\ell \in H^\ell$, we claim:
\[
\norm{Y^b_{j} g_\ell - \prod\limits_{s=1}^{j}\left (1 - \delta^{i_s+\ell-2s}_{i_s-2s} \right )Y_0^bg_\ell} \leq \norm{g_\ell}\sum\limits_{s=1}^j \gamma_{i_s-2s}^{i_s+\ell-2s} \prod\limits_{t=1}^{s-1}\left (1 - \delta^{i_t+\ell-2t}_{i_t-2t} \right ),
\]
which implies the result (notice that the indices $i_s$ shift by $b=k-\ell$). The base case $j=0$ is trivial. Assume the inductive hypothesis holds for all $Y^b_i, i<j$. By \Cref{lemma:DU-UD} and recalling $g_{\ell}\in \ker(D_{\ell})$, we have:
\begin{align*}
    Y^b_j g_\ell = \left(1-\delta^{i_1+\ell-2}_{i_1-2}\right)Y^{b}_{j-1}g_{\ell}+\Gamma g_\ell,
\end{align*}
where $\Gamma$ has spectral norm
\[
\norm{\Gamma} \leq \gamma^{i_1+\ell-2}_{i_1-2}.
\]
Notice that $Y^{b}_{j-1}$ has down operator indices $\{i_2-2,\ldots,i_j-2\}$.
The inductive hypothesis then implies:
\begin{align*}
        Y^b_j g_\ell &= \left(1-\delta^{i_1+\ell-2}_{i_1-2}\right)\prod\limits_{s=2}^{j}\left(1-\delta^{i_s+\ell-2s}_{i_s-2s}\right)Y^b_0g_\ell + \left(1-\delta^{i_1+\ell-2}_{i_1-2}\right)\Gamma'g_{\ell} + \Gamma g_\ell\\
        &= \prod\limits_{s=1}^{j}\left(1-\delta^{i_s+\ell-2s}_{i_s-2s}\right)g_\ell + \left(1-\delta^{i_1+\ell-2}_{i_1-2}\right)\Gamma'g_{\ell} + \Gamma g_\ell,
\end{align*}
where $\Gamma'g_{\ell}$ has norm
\[
\norm{\Gamma'g_{\ell}} \leq \norm{g_\ell}\sum\limits_{s=2}^{j}\gamma^{i_s+\ell-2s}_{i_s-2s}\prod_{t=2}^{s-1}\left ( 1- \delta^{i_t+\ell-2t}_{i_t-2t}\right ).
\]
Thus we may bound the norm of the righthand error term by:
\begin{align*}
\norm{\left(1-\delta^{i_1+\ell-2}_{i_1-2}\right)\Gamma'g_{\ell}+\Gamma g_\ell}
& \leq \left(1-\delta^{i_1+\ell-2}_{i_1-2}\right)\norm{\Gamma'}\norm{g_{\ell}} + \norm{\Gamma}\norm{g_\ell}\\
&\leq \sum\limits_{s=1}^j \gamma_{i_s-2s}^{i_s+\ell-2s} \prod\limits_{t=1}^{s-1}\left (1 - \delta^{i_t+\ell-2t}_{i_t-2t} \right )\norm{g_\ell},
\end{align*}
as desired. Recalling the shift in $i_s$ by $k-\ell$, we can then bound the resulting error by $(j+k)j\gamma\norm{g_\ell}$ since $\delta \in [0,1]^{d-1}$.
\end{proof}
It is worth noting that when $\gamma=0$, this implies that the HD-Level-Set decomposition is a true eigendecomposition. Since balanced walks are simply affine combinations of pure walks, this immediately implies a similar result for the more general case. To align with our definition of approximate eigendecompositions and \Cref{thm:approx-ortho}, we'll also need the following general relation between $\norm{g_\ell}$ and $\norm{f_\ell}$ for eposets proved in \cite{dikstein2018boolean} (albeit without the exact parameter dependence).
\begin{lemma}[{\cite[Lemma 8.11]{dikstein2018boolean}}]\label{lemma:fvsg-body}
Let $(X,\Pi)$ be a $d$-dimensional $(\delta,\gamma)$-eposet, $0 \leq \ell \leq k < d$, and let
\[
\rho^k_\ell = \prod\limits^{k-\ell}_{i=1} \left(1-\delta_{k-\ell-i}^{k-i}\right), \ \ \rho_{\text{min}} = \min_{0 \leq \ell \leq k}\{\rho^k_\ell\}.
\]
Then for any $f_\ell=U^{k}_{\ell}g_{\ell}$ for $g_\ell \in \text{Ker}(D_\ell)$ we have:
\[
\langle f_\ell, f_\ell \rangle \in (\rho^k_\ell \pm k^2\gamma)\langle g_\ell, g_\ell \rangle,
\]
and for all $i \neq \ell$:
\[
\langle f_\ell, f_i \rangle \leq O\left(\frac{k^2}{\rho_{\text{min}}}\gamma \norm{f_\ell}\norm{f_i} \right).
\]
\end{lemma}
As an aside, we remark that the parameter $\rho^k_\ell$ turns out to be a crucial throughout much of our work, and while it is difficult to interpret on general eposets, we prove it has a very natural form as long as non-laziness holds.
\begin{claim}[$\rho^k_\ell$ for regular eposets]\label{claim:reg}
Let $(X,\Pi)$ be a regular, $\gamma$-non-lazy\footnote{One can prove this claim more generally for any $\beta$-non-laziness, but most $\gamma$-eposets of interest are additionally $\gamma$-non-lazy, so this simplified version is generally sufficient.} $d$-dimensional $(\delta,\gamma)$-eposet. Then for any $i \leq k < d$, we have:
\[
\rho^k_i \in \frac{1}{R(k,i)} \pm err,
\]
where $err \leq O\left(\frac{i^3k^2R_{\text{max}}}{\delta_i(1-\delta_{i-1})}\gamma\right)$. Likewise as long as $\gamma \leq O\left(\frac{\max_i\{\delta_i(1-\delta_{i-1})\}}{i^3k^2R^2_{\text{max}}}\right)$ we have
\[
\rho_{\text{min}}^{-1} \leq O(R_{\text{max}}),
\]
where $R_{\text{max}} \coloneqq \max_{0 \leq i \leq k}\{R(k,i)\}$.
\end{claim}
This gives a nice generalization of the interpretation of $\rho^k_i$ on hypergraphs, which is well known to be $\frac{1}{{k \choose i}}$ \cite{dikstein2018boolean}.
We prove this claim in \Cref{app:regularity}. For simplicity, we will assume throughout the rest of this work that our eposets are $\gamma$-non-lazy, which is true for most cases of interest (see \Cref{app:regularity}). All results holds in the more general case using $\rho^k_i$ unless otherwise noted.

Combining \Cref{prop:pure-eig-vals} and \Cref{lemma:fvsg-body} immediately implies that the HD-Level-Set Decomposition is an approximate eigendecomposition in the sense of \Cref{def:approx-eigen}.
\begin{corollary}\label{cor:balanced-eig}
Let $(X,\Pi)$ be a $(\delta,\gamma)$-eposet and let $M=\sum\limits_{Y \in \mathcal{Y}} \alpha_Y Y$ be an HD-walk. For $1 \leq \ell \leq k$, if $f_\ell=U^{k}_\ell g_\ell$ for some $g_\ell \in H^\ell$, then for 
$\gamma \leq O\left(\frac{\max_i\{\delta_i(1-\delta_{i-1})\}}{k^5R^2_{\text{max}}}\right)$:
\[
\norm{Mf_\ell - \left (\sum\limits_{Y \in \mathcal Y} \alpha_Y \lambda_{Y,\delta,\ell}\right )f_\ell} \leq c\gamma\norm{f_\ell},
\]
where $\lambda_{Y,\delta,\ell}$ is the corresponding eigenvalues of the pure balanced walk $Y$ on a $(\delta,0)$-eposet (the form of which are given in \Cref{prop:pure-eig-vals}), and $c \leq O\left((h( M)+k)h(M)R(k,\ell)w(M)\right)$.
\end{corollary}
Thus as long as the walk in question is self-adjoint (e.g. canonical or swap walk), \Cref{thm:approx-ortho} immediately implies that the true spectrum is concentrated around these approximate eigenvalues.

Before moving on it is instructive (and as we will soon see quite useful) to give an example application of \Cref{cor:balanced-eig} to a basic higher order random walk. 
\begin{corollary}[Spectrum of Lower Canonical Walks]\label{cor:lower}
Let $(X,\Pi)$ be a $(\delta,\gamma)$-eposet. The approximate eigenvalues of the canonical lower walk $\widecheck{N}_k^{k-\ell}$ are:
\[
\lambda_j(\widecheck{N}_k^{k-\ell}) = \prod_{s=1}^{k-\ell}(1-\delta^{k-s}_{k-s-j}).
\]
\end{corollary}
\begin{proof}
The lower canonical walk $\widecheck{N}_k^{k-\ell}=U^k_\ell D^k_\ell$ is of height $k-\ell$, and has down operator at positions $\{1,\ldots,k-\ell\}$. In the language of \Cref{prop:pure-eig-vals} we therefore have $i_s=s$, which therefore gives:
\[
\lambda_j(\widecheck{N}_k^{k-\ell}) = \prod_{s=1}^{k-\ell}(1-\delta^{k-s}_{k-s-j}).
\]
Note this is $0$ when $j>\ell$.
\end{proof}
Similar to the case of $\rho^k_i$, while this is difficult to interpret in the general setting, the eigenvalues have a very natural form on non-lazy eposets given by the regularity parameters.
\begin{theorem}\label{claim:reg-lower}
Let $(X,\Pi)$ be a $\gamma$-non-lazy $(\delta,\gamma)$-eposet. The approximate eigenvalues of the canonical lower walk $\widecheck{N}_k^{k-i}$ are:
\[
\lambda_j(\widecheck{N}_k^{k-i}) \in \frac{R(i,j)}{R(k,j)} \pm c\gamma,
\]
where $c \leq O\left(\frac{i^4k^2R_{\text{max}}}{\delta_i(1-\delta_{i-1})}\gamma\right)$.
\end{theorem}
The proof requires machinery developed in \Cref{sec:expansion} and \Cref{app:regularity}, and is given in \Cref{app:regularity}.
\section{Pseudorandomness and the HD-Level-Set Decomposition }\label{sec:pseudorandomness}
Now that we know the spectral structure of HD-walks, we shift to studying their combinatorial structure. In particular, we will focus on how natural notions of pseudorandomness control the projection of functions onto the HD-Level-Set Decomposition.

Before proceeding, we state a simple corollary of Lemma \ref{lemma:fvsg-body} that will prove useful going forward:

\begin{corollary}
\label{cor:DDFK-cor}
Let $(X,\Pi)$ be a $(\delta,\gamma)$-eposet and suppose $f\in C_k$ has HD-Level-Set Decomposition $f=f_0+\ldots+f_k$. If $\gamma\leq \frac{c'\rho_{\min}}{k^3}$ for a sufficiently small constant $c'>0$, then 
\begin{equation}
\label{eq:norm-sum}
    \sum_{j=0}^k \|f_j\|\leq O(\sqrt{k}\|f\|).
\end{equation}

Moreover, for any subset of indices $I$, it holds that
\begin{equation*}
    -\sum_{j\in I} \langle f,f_j\rangle \leq O\left(\frac{k^3\gamma \|f\|^2}{\rho_{min}}\right).
\end{equation*}
In particular, if $I=\{j: \langle f,f_j\rangle\leq 0\}$, then
\begin{equation*}
    \sum_{j\in I} \vert\langle f,f_j\rangle\vert \leq O\left(\frac{k^3\gamma \|f\|^2}{\rho_{min}}\right).
\end{equation*}
\end{corollary}
\begin{proof}
For the first claim, recall that by the approximate orthogonality of the HD-Level-Set Decomposition (Lemma~\ref{lemma:fvsg-body}), we have for all $i \neq j$:
\[
|\langle f_i, f_j \rangle| \leq O\left(\frac{k^2}{\rho_{\text{min}}}\gamma\norm{f_i}\norm{f_j}\right).
\]
Then, applying Cauchy-Schwarz gives:
\begin{align*}
\left (\sum\limits_{j=1}^k \norm{f_j}\right)^2 &\leq k\sum\limits_{j=1}^k \norm{f_j}^2\\
&\leq k\langle f,f \rangle -k \sum\limits_{i \neq j \neq 0} \langle f_i,f_j \rangle\\
&\leq k\langle f,f \rangle + c\gamma\sum\limits_{i \neq j \neq 0}\norm{f_i}\norm{f_j}\\
&\leq k\langle f,f \rangle + c\gamma\left (\sum\limits_{j=1}^k \norm{f_j}\right)^2
\end{align*}
where $c \leq O\left(\frac{k^3}{\rho_{\text{min}}}\right)$. By our assumption on $\gamma$, we have $c\gamma \leq \frac{1}{2}$, and therefore rearranging yields
\begin{equation*}
    \sum\limits_{i=1}^k \norm{f_j} \leq O(\sqrt{k}\norm{f}).
\end{equation*}

We now show how the second claim is a consequence of the first. For any subset $I$, we have
\begin{align*}
    -\sum_{j\in I} \langle f,f_j\rangle&\leq -\sum_{j\in I}\sum_{i\neq j}\langle f,f_j\rangle\\
    &\leq \frac{Ck^2\gamma}{\rho_{\min}} \sum_{i,j} \norm{f_i}\norm{f_j}\\
    &=\frac{O(k^2\gamma)}{\rho_{\min}} \left(\sum_{i=0}^k \norm{f_i}\right)^2\\
    &\leq O\left(\frac{k^3\gamma \norm{f}}{\rho_{\min}}\right).
\end{align*}
\end{proof}
\subsection{$\ell_2$-pseudorandomness}
We start with pseudorandomness in the $\ell_2$-regime, which measures the variance of a set across links.
\begin{definition}[$\ell_2$-Pseudorandom functions \cite{bafna2020high}]\label{def:pseudorandom}
A function $f \in C_k$ is $(\varepsilon_1,\ldots,\varepsilon_\ell)$-$\ell_2$-pseudorandom if its variance across $i$-links is small for all $1 \leq i \leq \ell$:
\[
\mathrm{Var}(D^k_if) \leq \varepsilon_i |\mathbb{E}[f]|.
\]
\end{definition}
In their work on simplicial complexes,
BHKL \cite{bafna2020high} observed a close connection between $\ell_2$-pseudorandomness, the HD-Level-Set Decomposition, and the spectra of the lower canonical walks. We'll show the same connection holds in general for eposets.
\begin{theorem}\label{lem:low-level-weight}
Let $(X,\Pi)$ be a $(\delta,\gamma)$-eposet with $\gamma \leq O\left(\frac{\max_i\{\delta_i(1-\delta_{i-1})\}}{k^5R^2_{\text{max}}}\right)$. If $f \in C_k$ has HD-Level-Set Decomposition $f=f_0+\ldots+f_k$, then for any $\ell \leq k$, $\mathrm{Var}(D^k_\ell f)$ is controlled by its projection onto $V_k^0 \oplus \ldots \oplus V_k^\ell$ in the following sense: 
\begin{align*}
     \mathrm{Var}(D^k_\ell f)
     &\in \sum\limits_{j=1}^\ell \lambda_j(\widecheck{N}_k^{k-\ell})\langle f, f_j \rangle \pm c_k\gamma \|f\|^2,
\end{align*}
where $c_k \leq O(k^{5/2}R_{\text{max}})$ and $\lambda_j(\widecheck{N}_k^{k-\ell}) = \prod_{s=1}^{k-\ell}(1-\delta^{k-s}_{k-s-j})$.
\end{theorem}
\begin{proof}
To start, notice that since $\langle D^k_\ell f, D^k_\ell f \rangle = \langle \widecheck{N}_k^{k-\ell} f, f \rangle$ it is enough to analyze the application of $\widecheck{N}_k^{k-\ell}$ to $f$. By \Cref{cor:lower}, we know that each $f_j$ is an approximate eigenvector satisfying:
\[
\norm{\widecheck{N}_k^{k-\ell} f_j - \lambda_j(\widecheck{N}_k^{k-\ell}) f_j} \leq O(k^2R(k,\ell)\gamma)\norm{f_j},
\]
where $\lambda_j(\widecheck{N}_k^{k-\ell}) = 0$ for $j > \ell$. Combining these observations gives:
\begin{align*}
\text{Var}(D^k_\ell f) &= \left \langle D^k_\ell f,D^k_\ell f \right \rangle - \mathbb{E}[D^k_\ell f]^2\\
&= \left \langle f,U^k_\ell D^k_\ell f \right \rangle - \langle f,f_0 \rangle\\
    &= \sum\limits_{j=1}^k \langle f, U^k_\ell D^k_\ell f_j \rangle\\
    &\in \sum\limits_{j=1}^\ell\lambda_j(\widecheck{N}_k^{k-\ell})\langle f, f_j \rangle \pm O\left(\frac{k^2}{\rho_{\text{min}}}\gamma \norm{f}\sum\limits_{j=1}^k \norm{f_j}\right).
\end{align*}
where we have additionally used the fact that $\langle f, f_0 \rangle = \mathbb{E}[f]^2 = \mathbb{E}[D^k_\ell f]^2$ and $\lambda_0(\widecheck{N}_k^{k-\ell})=1$. Applying \Cref{eq:norm-sum} from Corollary~\ref{cor:DDFK-cor} to bound the sum in the error term 
and replacing $\rho$ with the relevant regularity parameters by \Cref{claim:reg} then gives the result.
\end{proof}
As an immediate corollary, we get a level-$i$ inequality for pseudorandom functions.
\begin{corollary}\label{cor:ell2-proj}
Let $(X,\Pi)$ be a $(\delta,\gamma)$-eposet with $\gamma \leq O\left(\frac{\max_i\{\delta_i(1-\delta_{i-1})\}}{k^5R^2_{\text{max}}}\right)$ and let $f \in C_k$ be an $(\varepsilon_1,\ldots,\varepsilon_\ell)$-$\ell_2$-pseudorandom function. Then for any $1 \leq i \leq \ell$:
\[
|\langle f, f_i \rangle| \leq R(k,i) \varepsilon_i|\mathbb{E}[f]| + c\gamma\| f\|^2,
\]
where $c \leq O\left(\frac{k^5R_{\text{max}}^2}{\max_i\{\delta_i(1-\delta_{i-1})\}}\right)$.
\end{corollary}
\begin{proof}
By Corollary~\ref{cor:DDFK-cor}, for any given $1\leq i\leq k$, it holds that $-\sum_{j\neq i}\langle f, f_j \rangle \leq O\left(\frac{k^3}{\rho_{\text{min}}}\gamma\| f\|^2\right)$. It follows from \Cref{lem:low-level-weight} that for all $0 \leq i \leq k$, the variance of $D_i^k f$ is lower bounded by its projection onto $f_i$:
\begin{align*}
    \text{Var}(D_i^k f) \geq \lambda_i(\widecheck{N}_k^{k-i}) \langle f,f_i \rangle - c\gamma\langle f,f \rangle,
\end{align*}
where $c \leq O(\frac{k^3}{\rho_{\text{min}}})$. Noting that $\lambda_i(\widecheck{N}_k^{k-i})=\rho^k_i$, if $i\leq \ell$, re-arranging the above and applying the pseudorandomness assumption gives:
\begin{align*}
    \langle f,f_i \rangle &\leq \frac{1}{\rho^k_i}\text{Var}(D_i^k f) + c_2\gamma\langle f,f \rangle\\
    &\leq \frac{1}{\rho^k_i}\varepsilon_i|\mathbb{E}[f]| + c_2\gamma\langle f,f \rangle,
\end{align*}
where $c_2 \leq O(\frac{k^3}{\rho^2_{\text{min}}})$. The lower bound on $\langle f,f_i \rangle$ is immediate from Corollary~\ref{cor:DDFK-cor} with the set $I=\{i\}$. Applying \Cref{claim:reg} then gives the result.
\end{proof}
As mentioned previously, this also recovers the tight inequality for simplicial complexes given in \cite{bafna2020high} where $R(k,i)={k \choose i}$, as well as providing the natural $q$-analog for $q$-simplicial complexes where $R(k,i)={k \choose i}_q$. 

\subsection{$\ell_\infty$-pseudorandomness}
While $\ell_2$-pseudorandomness is useful in its own right (e.g. for local-to-global algorithms for unique games \cite{bafna2020playing,bafna2020high}), there is also significant interest in a stronger $\ell_\infty$-variant in the hardness of approximation literature \cite{khot2018small,subhash2018pseudorandom}.
\begin{definition}[$\ell_\infty$-Pseudorandom functions]\label{def:pseudorandom-infty}
A function $f \in C_k$ is $(\varepsilon_1,\ldots,\varepsilon_\ell)$-$\ell_\infty$-pseudorandom if for all $1 \leq i \leq \ell$ its local expectation is close to its global expectation:
\[
\left \| D^k_if - \mathbb{E}[f] \right \|_\infty \leq \varepsilon_i.
\]
\end{definition}
In their recent work on $\ell_2$-structure of expanding simplicial complexes, BHKL prove a basic reduction from $\ell_\infty$ to $\ell_2$-pseudorandomness that allows for an analogous level-$i$ inequality for this notion as well. Here, we'll show the same result holds for general eposets. As in their work, we'll take advantage of a weak local-consistency property called locally-constant sign.
\begin{definition}[locally-constant sign \cite{bafna2020high}]
Let $(X,\Pi)$ be a graded poset. We say a function $f \in C_k$ has $\ell$-local constant sign if:
\begin{enumerate}
    \item $\mathbb{E}[f] \neq 0$,
    \item $\forall s \in X(\ell)$ s.t.\ $~\underset{X_s}{\mathbb{E}}[f] \neq 0: \text{sign}\left (\underset{X_s}{\mathbb{E}}[f]\right ) = \text{sign} \left (\mathbb{E}[f] \right )$.
\end{enumerate}
\end{definition}
With this in mind, we now state $\ell_\infty$-variant of \Cref{cor:ell2-proj}:
\begin{theorem}\label{thm:body-local-spec-proj}
Let $(X,\Pi)$ be a $(\delta,\gamma)$-eposet with $\gamma \leq O\left(\frac{\max_i\{\delta_i(1-\delta_{i-1})\}}{k^5R^2_{\text{max}}}\right)$ and let $f \in C_k$ have HD-Level-Set Decomposition $f=f_0+\ldots+f_k$. If $f$ is $(\varepsilon_1,\ldots,\varepsilon_\ell)$-$\ell_\infty$-pseudorandom, then for all $1 \leq i \leq \ell$:
\[
|\langle f, f_i \rangle| \leq  \left(R(k,i) + c\gamma\right)\varepsilon_i^2 + c\gamma\norm{f}^2,
\]
and if $f$ has $i$-local constant sign:
\[
|\langle f, f_i \rangle| \leq  \left(R(k,i) + c\gamma\right)\varepsilon_i|\mathbb{E}[f]| + c\gamma\norm{f}^2
\]
where in both cases $c \leq O\left(\frac{k^5R_{\text{max}}^2}{\max_i\{\delta_i(1-\delta_{i-1})\}}\right)$.
\end{theorem}
We note that when $f$ is boolean, this bound simplifies to
\[
|\langle f, f_i \rangle| \leq \left(R(k,i)\varepsilon_i+c\gamma\right) \mathbb{E}[f],
\]
which we'll see in the next section is a particularly useful form for analyzing edge expansion. The proof of \Cref{thm:body-local-spec-proj} relies mainly on a reduction to the $\ell_2$-variant for functions with locally-constant sign. This reduction is almost exactly the same as in \cite{bafna2020high}, but we include it for completeness.
\begin{lemma}\label{lem:reduction}
Let $(X,\Pi)$ be a graded poset and $f \in C_k$ a $(\varepsilon_1,\ldots,\varepsilon_\ell)$-$\ell_\infty$-pseudorandom function with $i$-local constant sign for any $i \leq \ell$. Then $f$ is $(\varepsilon_1,\ldots,\varepsilon_\ell)$-$\ell_2$-pseudorandom.
\end{lemma}
\begin{proof}
As in \cite{bafna2020high}, the idea is to notice that locally constant sign allows us to rewrite $\norm{D^k_i f}_2^2$ as an expectation over some related distribution $P_{i}$:
\begin{align*}
    \frac{1}{\mathbb{E}[f]}\langle D^k_i f, D^k_i f \rangle  &= \sum\limits_{s \in X(i)} \pi_i(s)\left (\frac{1}{\mathbb{E}[f]}\sum\limits_{t \in X_s}\frac{\pi_k(t)f(t)}{\pi_k(X_s)}\right )D^k_i f(s)\\
      &= \sum\limits_{s \in X(i)} \left (\frac{1}{R(k,i)}\frac{\sum_{t \in X_s} \pi_k(t)f(t)}{\mathbb{E}[f]}\right )D^k_i f(s)\\
    &= \underset{P_{i}}{\mathbb{E}}[D^k_i f], \label{eq:exp-to-norm}
\end{align*}
where $P_i$ being a probability distribution follows from the locally-constant sign of $f$, and the second step follows from 
the fact that $\pi_k(X_s) = \sum_{t\in X_s}\pi_k(t)=R(k,i)\pi_i(s)$. The result then follows easily from averaging:
\[
\left |\frac{1}{\mathbb{E}[f]}\text{Var}(D^k_if)\right |
= \left |\underset{P_i}{\mathbb{E}}[D^k_i f] - \mathbb{E}[f] \right | \leq \|D^k_if - \mathbb{E}[f]\|_{\infty}.
\]
When $\mathbb{E}[f]>0$, the $\ell_\infty$-norm here may be replaced with maximum.
\end{proof}
The proof of \Cref{thm:body-local-spec-proj} now follows from reducing to the case of locally-constant sign. The argument is exactly as in the proof of \cite[Theorem 8.7]{bafna2020high}, but we include it for completeness.
\begin{proof}[Proof of \Cref{thm:body-local-spec-proj}]
We focus on the general bound, since the result for functions with locally constant sign is immediate from \Cref{lem:reduction} and \Cref{cor:ell2-proj}. The argument for general functions $f$ follows simply from noting that we can always shift $f$ to have locally constant sign. With this in mind, assume $\mathbb{E}[f]\geq 0$ for simplicity (the negative case is similar). Let $f'=f + (\varepsilon_i - \mathbb{E}[f])\mathbbm{1}$ be the aforementioned shift. As long as $\varepsilon_i>0$, it is easy to see that $f'$ has $i$-local constant sign and further that
\[
f' = f_0' + f_i + \ldots + f_k,
\]
where $f_0' = f_0+(\varepsilon_i - \mathbb{E}[f])\mathbbm{1}$. Since shifts have no effect on $\ell_\infty$-pseudorandomness, $f'$ is $(\varepsilon_1,\ldots, \varepsilon_\ell)$-$\ell_\infty$-pseudorandom by assumption, and therefore $(\varepsilon_1,\ldots, \varepsilon_\ell)$-$\ell_2$-pseudorandom by \Cref{lem:reduction}. We can now apply \Cref{cor:ell2-proj} to get:
\begin{align*}
\langle f +(\varepsilon_i - \mathbb{E}[f])\mathbbm{1}, f_i \rangle &\leq \frac{1}{\rho^k_i}\varepsilon_i\mathbb{E} [f+(\varepsilon_i - \mathbb{E}[f])\mathbbm{1}] + c\gamma \langle f + (\varepsilon_i - \mathbb{E}[f])\mathbbm{1}, f + (\varepsilon_i - \mathbb{E}[f])\mathbbm{1} \rangle \\
&\leq \left (\frac{1}{\rho^k_i}+c\gamma\right)\varepsilon_i^2 + c\gamma\langle f, f \rangle,
\end{align*}
since $\langle f_i, \mathbbm{1} \rangle=0$ for all $i>0$. Finally, as this holds for all $\varepsilon_i>0$, a limiting argument implies the result for $\varepsilon_i = 0$. Applying \Cref{claim:reg} completes the proof.
\end{proof}

\section{Expansion of HD-walks}\label{sec:expansion}
It is well known that higher order random walks on simplicial complexes (e.g.\ the Johnson graphs) are not small-set expanders. BHKL gave an exact characterization of this phenomenon for local-spectral expanders: they showed that the expansion of any $i$-link with respect to an HD-walk $M$ is almost exactly $1-\lambda_i(M)$. Moreover, using the level-$i$ inequality from the previous section, BHKL proved a tight converse to this result in an $\ell_2$-sense: \textit{any} non-expanding set must have high variance across links. This gave a complete $\ell_2$-characterization of non-expanding sets on local-spectral expanders, and lay the structural groundwork for new algorithms for unique games over HD-walks.

In this section, we'll show that these results extend to general expanding posets. To start, let's recall the definition of edge expansion.
\begin{definition}[Weighted Edge Expansion]
Let $(X,\Pi)$ be a graded poset and $M$ a $k$-dimensional HD-Walk. The weighted edge expansion of a subset $S \subset X(k)$ with respect to $M$ is
\[
\Phi(S) = \underset{v \sim \pi_k|_S}{\mathbb{E}}\left[ M(v,X(k) \setminus S)\right],
\]
where 
\[
M(v, X(k) \setminus S) = \sum\limits_{y \in X(k) \setminus S} M(v,y)
\]
and $M(v,y)$ denotes the transition probability from $v$ to $y$.
\end{definition}
Before we prove the strong connections between links and expansion, we need to introduce an important property of HD-walks, monotonic eigenvalue decay.
\begin{definition}[Monotonic HD-walk]
Let $(X,\Pi)$ be a $(\delta,\gamma)$-eposet. We call an HD-walk $M$ monotonic if its approximate eigenvalues $\lambda_i(M)$ (given in \Cref{cor:balanced-eig}) are non-increasing.
\end{definition}
Most HD-walks of interest (e.g.\ pure walks, partial-swap walks on simplicial or $q$-simplicial complexes, etc.) are monotonic. This property will be crucial to understanding expansion. To start, let's see how it allows us to upper bound the expansion of links.
\begin{theorem}[Local Expansion vs Global Spectra]\label{thm:link-lower}
Let $(X,\Pi)$ be a $(\delta,\gamma)$-eposet and $M$ be a $k$-dimensional monotonic HD-walk.
Then for all $0 \leq i \leq k$ and $\tau \in X(i)$:
\[
\Phi(X_\tau) \in 1 - \lambda_{i}(M) \pm c\gamma,
\]
where $c \leq O\left(\frac{k^5R^2_{\text{max}}(h(M)+k)h(M)w(M)}{\delta^k_{k-i}(1-\delta_{i-1})}\right)$.
\end{theorem}
The key to proving \Cref{thm:link-lower} is to show that the weight of an $i$-link lies almost entirely on level $i$ of the HD-Level-Set Decomposition. To show this, we'll rely another connection between regularity and eposet parameters for non-lazy posets. 
\begin{claim}\label{claim:regularity2}
Let $(X,\Pi)$ be a $d$-dimensional $(\delta,\gamma)$-eposet. 
Then for every $1 \leq k \leq d$
and $0 \leq i \leq k$, the following approximate relation between the eposet and regularity parameters holds:
\[
\lambda_i(N_k^1) \in  \frac{R(k,i)}{R(k+1,i)} \pm \left(\gamma^k_{k-i} + R(k,i)\delta^k_{k-i}\gamma\right)
\]
where we recall $\lambda_i(N_k^1)=1-\prod\limits^{k}_{j=i} \delta_j$.
\end{claim}
We prove this relation in \Cref{app:regularity}. With this in hand, we can show links project mostly onto their corresponding level.

\begin{lemma}\label{lemma:link-projection}
Let $(X,\Pi)$ be a $d$-dimensional $(\delta,\gamma)$-eposet with $\gamma \leq O\left(\frac{\max_i\{\delta_i(1-\delta_{i-1})\}}{k^5R^2_{\text{max}}}\right)$. Then for all $0 \leq i \leq k < d$ and $\tau \in X(i)$, $\id{X_\tau}$ lies almost entirely in $V_k^i$. That is for all $j \neq i$:
\[
\left |\frac{\langle \id{X_{\tau},i}, \id{X_\tau,j} \rangle}{\langle \id{X_\tau}, \id{X_\tau}\rangle}\right| \leq O\left(\frac{k^3R_{\text{max}}}{\delta^k_{k-i}(1-\delta_{i-1})} \gamma\right).
\]
\end{lemma}
\begin{proof}
We'll show that the expansion of $\id{X_\tau}$ with respect to the upper walk $N_k^1$ is almost exactly $1-\lambda_i(N_k^1)$, which implies most of the weight must lie on $V^k_i$. We'll start by analyzing the expansion of $\id{X_\tau}$ through a simple combinatorial argument. First, since $D$ and $U$ are adjoint we have:
\begin{align*}
    \bar{\Phi}(\id{X_\tau}) &= \frac{\langle \id{X_\tau}, D_{k+1}U_k \id{X_\tau} \rangle }{\langle \id{X_\tau}, \id{X_\tau} \rangle}\\
    &=\frac{\langle U_k\id{X_\tau}, U_k \id{X_\tau} \rangle }{\langle \id{X_\tau}, \id{X_\tau} \rangle}.
\end{align*}
The trick is now to notice that ${\langle \id{X_\tau}, \id{X_\tau} \rangle} = R(k,i)\pi_i(\tau)$, and  $\langle U_k\id{X_\tau}, U_k \id{X_\tau} \rangle = \frac{R(k,i)^2}{R(k+1,i)}\pi_i(\tau)$. As a result, applying \Cref{claim:regularity2} gives:
\[
 \bar{\Phi}(\id{X_\tau}) \in \lambda_i(N_k^1) \pm (c\gamma + R(k,i)\gamma),
\]
for $c \leq k\gamma$. To see why this implies that most of the weight lies on $V^k_i$, note that we can also unfold the expansion of $\id{X_\tau}$ in terms of the HD-Level-Set decomposition:
\begin{align*}
\bar{\Phi}(\id{X_\tau}) &= \frac{1}{\langle \id{X_\tau}, \id{X_\tau} \rangle }\sum\limits_{j=0}^{i} \langle \id{X_\tau}, N^1_k \id{X_\tau,j} \rangle\\
&\in \frac{1}{\langle \id{X_\tau}, \id{X_\tau} \rangle }\sum\limits_{j=0}^{i} \lambda_i(N^1_k)\langle \id{X_\tau}, \id{X_\tau,j} \rangle \pm c_2\gamma
\end{align*}
where $c_2 \leq \frac{k\sqrt{k}}{\rho_{\text{min}}}$.
Recall from Corollary~\ref{cor:DDFK-cor} that for the set $I$ of indices with negative inner product, it holds that $-\sum_{j\in I}\langle \id{X_\tau}, \id{X_\tau,j} \rangle \leq O\left(\frac{k^3}{\rho_{\text{min}}}\gamma \langle  \id{X_\tau},  \id{X_\tau} \rangle\right)$.  Moreover, the positive inner products (i.e. the indices not in $I$) must sum to at least $\langle  \id{X_\tau},  \id{X_\tau} \rangle$. Then if there exists some $j \neq i$ such that $\langle \id{X_\tau}, \id{X_\tau,j} \rangle >c_3 \langle \id{X_\tau}, \id{X_\tau} \rangle$ for large enough $c_3 \leq O\left(\frac{1}{\delta^k_{k-i}(1-\delta_{i-1})} \cdot \left(\frac{k^3}{\rho_{\text{min}}}\gamma + R(k,i)\gamma\right)\right)$, the non-expansion would be strictly larger than $\lambda_i(N_k^1) + c\gamma +R(k,i)\gamma$ giving the desired contradiction (note that $(1-\delta_{i-1})\delta^k_{k-1}$ is the gap between the $i-1$st and $i$th approximate eigenvalue). The form in the theorem statement then follows from applying \Cref{claim:reg}.
\end{proof}
We note that the above is the only result in our work that truly relies on non-laziness (it is used only to replace $\rho$ with regularity in all other results). It is possible to recover the upper bound in \Cref{thm:link-lower} for general eposets via arguments used in \cite{bafna2020high}, but the lower bound remains open for concentrated posets. With that in mind, we now prove \Cref{thm:link-lower}.
\begin{proof}[Proof of \Cref{thm:link-lower}]
By the previous lemma, we have 
\[
\left |\frac{\langle \id{X_\tau}, \id{X_\tau,j} \rangle}{\langle \id{X_\tau}, \id{X_\tau}\rangle}\right| \leq O\left(\frac{1}{\delta^k_{k-i}(1-\delta_{i-1})} \cdot \left(\frac{k^3}{\rho_{\text{min}}}\gamma + R(k,i)\gamma\right)\right).
\]
Expanding out $\bar{\Phi}(\id{X_\tau})$ then gives:
\begin{align*}
\bar{\Phi}(\id{X_\tau}) &= \frac{1}{\langle \id{X_\tau}, \id{X_\tau} \rangle }\sum\limits_{j=0}^{i} \langle \id{X_\tau}, M \id{X_\tau,j} \rangle\\
&\leq \frac{1}{\langle \id{X_\tau}, \id{X_\tau} \rangle }\sum\limits_{j=0}^{i} \lambda_i(M)\langle \id{X_\tau}, \id{X_\tau,j} \rangle + c_2\gamma\\
&\leq \lambda_i(M)\frac{\langle \id{X_\tau}, \id{X_\tau,i} \rangle}{\langle \id{X_\tau}, \id{X_\tau} \rangle} + err_1\\
&\leq \lambda_i(M) + err_2.
\end{align*}
where $c_2,err_1,err_2 \leq O\left(\frac{k}{\delta^k_{k-i}(1-\delta_{i-1})} \left(\frac{k^2(h(M)+k)h(M)w(M)}{\rho_{\text{min}}}\gamma + R(k,i)\gamma\right)\right)$ and the last step follows from the approximate orthogonality. As usual, the form in the theorem statement then follows from applying \Cref{claim:reg}.
\end{proof}

Altogether, we've seen that for sufficiently nice expanding posets, the expansion of any $i$-link with respect to an HD-walk is almost exactly $1-\lambda_i(M)$. Since HD-walks are generally poor expanders (have large $\lambda_1(M)$), \Cref{thm:link-lower} implies that links are examples of small, non-expanding sets. Following BHKL, we'll now prove a converse to this result: any non-expanding set must be explained by some structure inside links. To help give a precise statement, we first recall BHKL's notion of Stripped Threshold Rank (specialized to eposets for convenience).
\begin{definition}[Stripped Threshold Rank \cite{bafna2020high}]
Let $(X,\Pi)$ be a $(\delta,\gamma)$-eposet and $M$ a $k$-dimensional HD-walk with $\gamma$ small enough that \Cref{thm:approx-ortho} implies the HD-Level-Set Decomposition has a corresponding decomposition of disjoint eigenstrips $C_k=\bigoplus W_k^i$. The ST-Rank of $M$ with respect to $\eta$ is the number of strips containing an eigenvector with eigenvalue at least $\eta$:
\[
R_\eta(M) = |\{ W_k^i : \exists f \in V^i, Mf = \lambda f, \lambda > \eta\}|.
\]
We often write just $R_\eta$ when $M$ is clear from context.
\end{definition}
With this in mind, we'll show a converse to \Cref{thm:link-lower} in both $\ell_2$ and $\ell_\infty$ senses (respectively that any non-expanding set must have high variance over links, and must be more concentrated than expected in some particular link). It is convenient to express these results through their contrapositives: that pseudorandom sets expand. The proof is the same as in \cite{bafna2020high} for simplicial complexes, but we include it for completeness.
\begin{theorem}\label{thm:hdx-expansion}
Let $(X,\Pi)$ $(\delta,\gamma)$-eposet, $M$ a $k$-dimensional, monotonic HD-walk, and $\gamma$ small enough that the eigenstrip intervals of \Cref{thm:approx-ortho} are disjoint. For any $\eta>0$, let $r=R_\eta(M)-1$. Then the expansion of a set $S \subset X(k)$ of density $\alpha$ is at least:
\[
\Phi(S) \geq 1 - \alpha - (1-\alpha)\eta - \sum\limits_{i=1}^{r} (\lambda_i(M)-\eta)R(k,i)\varepsilon_i - c\gamma
\]
where $S$ is $(\varepsilon_1,\ldots,\varepsilon_r)$-pseudorandom and $c \leq O\left(\frac{k^{5}R^2_{\text{max}}(h(M)+k)h(M)w(M)}{\max_i\{\delta_{i}(1-\delta_{i-1})\}}\right)$. 
\end{theorem}
\begin{proof}
Let $\mathbbm{1}_S=\mathbbm{1}_{S,0} + \ldots + \mathbbm{1}_{S,k}$ be the HD-Level-Set Decomposition of the indicator of $S$. By linearity of the inner product, we may then write:
\begin{align*}
\Phi(S) &= 1 - \frac{1}{\mathbb{E}[\mathbbm{1}_{S}]}\langle \mathbbm{1}_{X_\tau}, M \mathbbm{1}_{X_\tau} \rangle\\
&= 1 - \frac{1}{\mathbb{E}[\mathbbm{1}_{S}]}\sum\limits_{j=0}^k \langle \mathbbm{1}_{S}, M \mathbbm{1}_{S,j} \rangle\\
&= 1 - \frac{1}{\mathbb{E}[\mathbbm{1}_{S}]}\sum\limits_{j=0}^k \lambda_j(M)\langle \mathbbm{1}_{S}, \mathbbm{1}_{S,j} \rangle - \frac{1}{\mathbb{E}[\mathbbm{1}_{S}]}\sum\limits_{j=0}^k\langle \mathbbm{1}_{S}, \Gamma_j\mathbbm{1}_{S,j} \rangle
\end{align*}
 where $\norm{\Gamma_j} \leq O\left((h(M)+k)h(M)\frac{w(M)}{\rho_{\text{min}}}\right)$. The trick is now to notice we can bound the righthand error term using Cauchy-Schwarz:
 \begin{align*}
 \left |\frac{1}{\mathbb{E}[\mathbbm{1}_{S}]}\sum\limits_{j=0}^k\langle \mathbbm{1}_{S}, \Gamma_j\mathbbm{1}_{S,j} \rangle\right | &\leq \frac{1}{\mathbb{E}[\mathbbm{1}_{S}]}\sum\limits_{j=0}^k\left |\langle \mathbbm{1}_{S}, \Gamma_j\mathbbm{1}_{S,j} \rangle\right|\\
 &\leq \frac{1}{\mathbb{E}[\mathbbm{1}_{S}]}\sum\limits_{j=0}^k\norm{\Gamma_j}\norm{\mathbbm{1}_{S}}\norm{\mathbbm{1}_{S,j}}\\
 &\leq c\gamma\frac{\norm{\mathbbm{1}_{S}}}{\mathbb{E}[\mathbbm{1}_{S}]}\sum\limits_{j=0}^k\norm{\mathbbm{1}_{S,j}}\\
 &\leq c_1\gamma,
 \end{align*}
 where $c \leq O\left((h(M)+k)h(M)\frac{w(M)}{\rho_{\text{min}}}\right)$ and $c_1 \leq O(\sqrt{k}c)$ by \Cref{eq:norm-sum}.
Since $M$ is a monotonic walk, we can further write:
\begin{align*}
\Phi(S) &\geq 1- \frac{1}{\mathbb{E}[\mathbbm{1}_S]}\sum\limits_{i=0}^{r} \lambda_i(M)\langle \mathbbm{1}_{S}, \mathbbm{1}_{S,i} \rangle - \frac{1}{\mathbb{E}[\mathbbm{1}_S]} \sum\limits_{i=r+1}^k \lambda_i(M)\langle \mathbbm{1}_{S}, \mathbbm{1}_{S,i} \rangle - c_1\gamma
\\
&\geq 1- \frac{1}{\mathbb{E}[\mathbbm{1}_S]}\sum\limits_{i=0}^{r} \lambda_i(M)\langle \mathbbm{1}_{S}, \mathbbm{1}_{S,i} \rangle - \frac{\eta}{\mathbb{E}[\mathbbm{1}_S]} \sum\limits_{i=r+1}^k \langle \mathbbm{1}_{S}, \mathbbm{1}_{S,i} \rangle - c_2\gamma\\
&= 1 - \frac{1}{\mathbb{E}[\mathbbm{1}_S]}\sum\limits_{i=0}^{r} \lambda_i(M)\langle \mathbbm{1}_{S}, \mathbbm{1}_{S,i} \rangle - \eta\left(1 - \frac{1}{\mathbb{E}[\mathbbm{1}_S]}\sum\limits_{i=0}^{r} \langle \mathbbm{1}_{S}, \mathbbm{1}_{S,i} \rangle \right) - c_2\gamma\\
&= 1- \eta - \frac{1}{\mathbb{E}[\mathbbm{1}_S]}\sum\limits_{i=0}^{r} (\lambda_i(M)-\eta)\langle \mathbbm{1}_{S}, \mathbbm{1}_{S,i} \rangle  - c_2\gamma\\
&= 1- \eta - (1-\eta)\alpha - \frac{1}{\mathbb{E}[\mathbbm{1}_S]}\sum\limits_{i=1}^{r} (\lambda_i(M)-\eta)\langle \mathbbm{1}_{S}, \mathbbm{1}_{S,i} \rangle  - c_2\gamma,
\end{align*}
where $c_2 \leq O\left(k^2(h(M)+k)h(M)\frac{w(M)}{\rho_{\text{min}}}\right)$. To justify the second inequality, observe that for any $r<i\leq k$ such that $\langle \mathbbm{1}_S,\mathbbm{1}_{S,i}\rangle\geq 0$, replacing $\lambda_i(M)$ with $\eta$ is valid. For the set $I$ of $r<i\leq k$ with negative inner product, Corollary~\ref{cor:DDFK-cor} implies that the sum over $I$ is $O(k^3\gamma/\rho_{\min})$, so the inequality remains valid by absorbing the small error into $c_2$. Applying \Cref{cor:ell2-proj} to bound $\langle \mathbbm{1}_{S}, \mathbbm{1}_{S,i} \rangle$ then gives the $\ell_2$-variant result, \Cref{thm:body-local-spec-proj} gives the $\ell_\infty$-variant, and \Cref{claim:reg} gives the form given in the theorem statement.
\end{proof}
We note that \Cref{thm:hdx-expansion} recovers the analogous result for simplicial complex in \cite{bafna2020high} by plugging in the appropriate value $R(k,i) = {k \choose i}$. BHKL also prove this special case is tight in two senses. First, they show that if one wants to retain linear dependence on the pseudorandomness parameter $\varepsilon$, \Cref{thm:hdx-expansion} is tight in both the $\ell_2$ and $\ell_\infty$-regimes. Second, they show that the dependence on $k$ is necessary in the $\ell_2$-regime, even if we allow sub-optimal dependence on $\varepsilon$. In the next section, we'll generalize this result to $q$-simplicial complexes as well. In both cases the proofs are highly structural and depend on the underlying structure of the poset---it remains an interesting open problem whether this bound is tight for all poset structures.


\section{The Grassmann and $q$-eposets}\label{sec:q-HDX}
In this section, we examine the specification of our results on eposets to expanding subsets of the Grassmann poset. We show that our analysis is tight in this regime via a classic example of a small non-expanding set in the Grassmann graphs called co-links.

\subsection{Spectra}
We'll start by examining the spectrum of HD-walks on the Grassmann and $q$-eposets. We'll focus our attention in this section on the most widely used walks in the literature, the canonical and partial-swap walks. To start, recall that the Grassmann poset itself is sequentially differential with parameters
\begin{equation}\label{eq:grassmann-params}
\delta_i = \frac{(q^i-1)(q^{n-i+1}-1)}{(q^{i+1}-1)(q^{n-i}-1)}.
\end{equation}
Plugging this into \Cref{prop:pure-eig-vals} gives a nice exact form for the spectra of canonical walks.
\begin{corollary}[Grassmann Poset $N_k^j$ Spectra]
\label{cor:canon-q-HDX-spectra}
Let $X=G_q(n,d)$ be the Grassmann Poset, $k+j \leq d$, and $f_\ell=U_\ell^{k}g_\ell$ for some $g_\ell \in H^\ell$. Then:
\[
N_k^j f_\ell = \lambda_\ell f_\ell,
\]
where,
\[
\lambda_\ell = q^{\ell j}\frac{{k+j-\ell \choose j}_q}{{k+j \choose j}_q} \frac{{n-k-\ell \choose j}_q}{{n-k \choose j}_q} \approx q^{-\ell j}.
\]
\end{corollary}
\begin{proof}
By \Cref{prop:pure-eig-vals} we have that
\begin{align*}
\lambda_\ell(N^j_k) &=
\prod\limits_{s=1}^j \left (1 - \prod^{k-s+j}_{i=\ell} \delta_i \right )\\
&= \prod\limits_{s=1}^j \left (1 - \prod^{k-s+j}_{i=\ell} \frac{(q^i-1)(q^{n-i+1}-1)}{(q^{i+1}-1)(q^{n-i}-1)} \right ).
\end{align*}
The result then follows from telescoping the interior product and simplifying:
\begin{align*}
&= \prod\limits_{s=1}^j \left (1 - \frac{(q^\ell-1)(q^{n-\ell+1}-1)}{(q^{k-s+j+1}-1)(q^{n+s-k-j}-1)} \right )\\
&=q^{\ell j}\left(\prod\limits_{s=1}^j \frac{ \left(q^{k+j-s-\ell+1}-1\right) }{\left(q^{k+j-s+1}-1\right)}\right)\left(\prod\limits_{s=1}^j \frac{ \left(q^{n+s-k-j-\ell}-1\right)}{ \left(q^{n+s-k-j}-1\right)}\right)\\
&=q^{\ell j}\frac{{k+j-\ell \choose j}_q}{{k+j \choose j}_q} \frac{{n-k-\ell \choose j}_q}{{n-k \choose j}_q}
\end{align*}
as desired.
\end{proof}
This recovers a very simple proof of classical results to this effect (see e.g.\ \cite{delsarte1976association}). An analogous computation gives an approximate bound on the spectrum of $N^j_k$ on $q$-eposets as well.
\begin{corollary}[$q$-eposets $N_k^j$ Spectra]
Let $(X,\Pi)$ be a $d$-dimensional $\gamma$-$q$-eposet with $\gamma \leq q^{-\Omega(k^2)}$, $k+j \leq d$, and $f_\ell=U_\ell^{k-1}g_\ell$ for some $g_\ell \in H^\ell$. Then:
\begin{align*}
\norm{N_k^j f_\ell - \frac{{k+j - \ell \choose j}_q}{{k+j \choose j}_q}f_\ell} \leq O\left(j(j+k){k \choose \ell}_q\right)\gamma\norm{f_\ell}
\end{align*}
\end{corollary}
Note that for small enough $\gamma$, \Cref{thm:approx-ortho} implies that the true spectra is then concentrated around these values as well. It is also worth noting that these eigenvalues are, as one would expect, the natural $q$-analog of the corresponding eigenvalues on simplicial complexes.

It turns out that this fact will carry over to the important class of partial-swap walks as well. Partial-swap walks on simplicial complexes were originally analyzed by AJT \cite{alev2019approximating}, who showed they can be written as a hypergeometric combination of canonical walks. Their proof is specific to the structure of simplicial complexes, and some work is required to generalize their ideas to the Grassmann case. Following the overall proof strategy of AJT, it will be helpful to first show that the canonical walks themselves can be written as an expectation of swap walks over a $q$-hypergeometric distribution, and then use the $q$-binomial inversion theorem to derive the desired result.
\begin{lemma}[$q$-analog of  {\cite[Lemma 4.11]{alev2019approximating}}]\label{lemma:N-vs-S-q}
Let $(X,\Pi)$ be a pure, measured $q$-simplicial complex. Then:
\[
N_k^j = \sum\limits_{i=0}^j q^{i^2}\frac{{j \choose i}_q{k \choose k-i}_q}{{k+j \choose k}_q}S_k^i
\]
\end{lemma}
\begin{proof}
We follow the structure and notation of \cite[Lemma 4.11]{alev2019approximating}. Assume that the canonical walk starts at a subspace $V \in X(k)$, and walks up to $W \in X(k+j)$. We wish to analyze the probability that upon walking back down to level $k$, a subspace $V'$ with intersection $k-i$ is chosen, that is:
\[
dim(V \cap V')=k-i.
\]
Let such an event be denoted $\mathcal{E}_i(W)$. It follows from elementary $q$-combinatorics (see e.g.\ \cite[Lemma 9.3.2]{brouwer2012distance}) that
\[
\Pr_{V' \subset W}[\mathcal{E}_i(W)~|~W] = q^{i^2}\frac{{j \choose i}_q{k \choose k-i}_q}{{k+j \choose k}_q},
\]
where $V' \in X(k)$ is drawn uniformly from the $k$-dimensional subspaces of $W$. In essence, we wish to relate this process to the swap walk $S_k^i$. To do so, note that while the swap walk (as defined) only walks up to $X(k+i)$, walking up to $X(k+j)$ and conditioning on intersection $i$, a process called the \textit{$i$-swapping $j$-walk} by \cite{alev2019approximating}, is exactly the same due to symmetry (via the regularity condition, see \cite{alev2019approximating}[Proposition 4.9] for a more detailed explanation). Thus consider the $i$-swapping $j$-walk, and let $T'_i$ denote the variable standing for the subspace chosen by the walk. Conditioned on picking the same $W$ as the canonical walk in its ascent, we may relate $T'_i$ to the canonical walk:
\[
\Pr[T'_i = T~|~W] = \Pr[V'=T~|~W~\text{and}~\mathcal{E}_i(W)]
\]
We may now decompose the canonical walk by intersection size:
\begin{align*}
    N_k^j(V,T) &= \sum\limits_{i=0}^j \sum\limits_{W \in X(k+j)}\Pr[W]\Pr[\mathcal{E}_i(W)~|~W]\Pr[V'=T~|~W~\text{and}~\mathcal{E}_i(W)]\\
    &=\sum\limits_{i=0}^j \sum\limits_{W \in X(k+j)}q^{i^2}\frac{{j \choose i}_q{k \choose k-i}_q}{{k+j \choose k}_q}\underset{W \supset V}{\mathbb{E}}[\Pr[V'=T~|~W~\text{and}~\mathcal{E}_i(W)]]\\
    &=\sum\limits_{i=0}^j \sum\limits_{W \in X(k+j)}q^{i^2}\frac{{j \choose i}_q{k \choose k-i}_q}{{k+j \choose k}_q}\underset{W \supset V}{\mathbb{E}}[\Pr[T'_i=T~|~W~]]\\
    &=\sum\limits_{i=0}^j \sum\limits_{W \in X(k+j)}q^{i^2}\frac{{j \choose i}_q{k \choose k-i}_q}{{k+j \choose k}_q}\Pr[T_i'=T]\\
    &=\sum\limits_{i=0}^j \sum\limits_{W \in X(k+j)}q^{i^2}\frac{{j \choose i}_q{k \choose k-i}_q}{{k+j \choose k}_q}S_k^i(V,T)
\end{align*}
\end{proof}
This results in the $q$-analog of the analogous result on simplicial complexes \cite[Lemma 4.11]{alev2019approximating}. To recover the analogous statement writing partial-swap walks in terms of canonical walks, we can now apply a $q$-Binomial inversion theorem. 
\begin{lemma}[$q$-Binomial Inversion (Theorem 2.1 \cite{zou2017q})]\label{lemma:q-binom-inversion}
Suppose $\{a_i\}_{i\geq1}$, $\{b_i\}_{i\geq1}$ are two sequences. 
If:
\[
a_j = \sum\limits_{i=1}^{j}(-1)^i{j \choose i}_q b_i,
\]
then
\[
b_j = \sum\limits_{i=1}^{j} (-1)^iq^{{j-i \choose 2}}{j \choose i}_q a_i
\]
\end{lemma}
We note that \cite[Theorem 2.1]{zou2017q} is stated in slightly more generality in the original work, but the above lemma is an immediate application. With this in hand, we can finally prove that swap walks on the Grassmann poset are indeed HD-walks:
\begin{proposition}\label{prop:Grassmann-decomp}
Let $(X,\Pi)$ be a weighted pure $q$-simplicial complex. Then for $k+j \leq d$:
\[
 S_k^j = \frac{1}{q^{j^2}{k \choose k-j}_q}\sum\limits_{i=0}^{j} (-1)^{j-i}q^{{j-i \choose 2}}{j \choose i}_q{k+i \choose i}_q
 N_k^{i},
 \]
and similarly,
\[
 J_q(n,k,t) = S_k^{k-t}=\frac{1}{q^{(k-t)^2}{k \choose t}_q}\sum\limits_{i=0}^{k-t} (-1)^{k-t-i}q^{{k-t-i \choose 2}}{k-t \choose i}_q{k+i \choose i}_q
 N_k^{i}
 \]
\end{proposition}
\begin{proof}
The proof is an easy application of \Cref{lemma:q-binom-inversion} and the $q$-Binomial theorem. In particular, for any $V,V' \in X(k)$, let 
\[
a_i = (-1)^iq^{i^2}{k \choose k-i}_qS_k^i(V,V').
\]
Noting that $N_0^j=S_0^j=I$, \Cref{lemma:N-vs-S-q} gives the following equality:
\[
{k+j \choose k}_q\left (N_k^j(V,V')-\frac{1}{{k+j \choose k}_q}I(V,V')\right ) = \sum\limits_{i=1}^j (-1)^{i}{j \choose i}_qa_i.
\]
Setting the second sequence $\{b_i\}_{i\geq1}$ to
\[
b_i = {k+i \choose k}_q\left(N_k^i(V,V')-\frac{1}{{k+i \choose k}_q}I(V,V')\right ),
\]
\Cref{lemma:q-binom-inversion} then implies:
\begin{align*}
q^{j^2}{k \choose k-j}_qS_k^j(V,V') &= \sum\limits_{i=1}^j(-1)^{j-i}q^{{j-i \choose 2}}{k+i \choose k}\left(N_k^i(V,V')-\frac{1}{{k+i \choose k}_q}I(V,V')\right )\\
& = \sum\limits_{i=1}^j(-1)^{j-i}q^{{j-i \choose 2}}{k+i \choose k}N_k^i(V,V')-\sum\limits_{i=1}^j(-1)^{j-i}q^{{j-i \choose 2}}I(V,V')\\
& =  \sum\limits_{i=0}^j(-1)^{j-i}q^{{j-i \choose 2}}{k+i \choose k}N_k^i(V,V')
\end{align*}
where the last step follows from the $q$-Binomial theorem.
\end{proof}
Once again, we note that this is unsurprisingly the $q$-analog of the analogous statement on simplicial complexes (see \cite[Corollary 4.13]{alev2019approximating}). Finally, we'll use this to show that the eigenvalues of partial-swap walks on $q$-simplicial complexes are given by the natural $q$-analog of the simplicial complex case.
\begin{corollary}\label{cor:swap-q-HDX-spectra}
Let $(X,\Pi)$ be a $d$-dimensional $\gamma$-$q$-eposet with $\gamma$ sufficiently small, $k+j \leq d$, and $f_\ell = U^{k}_\ell g_\ell$ for some $g_\ell \in H^\ell$. Then:
\[
\norm{S_k^j f_\ell - \frac{{k-j \choose \ell}_q}{{k \choose \ell}_q}f_\ell} \leq O\left (\left(\frac{q}{q-1}\right)^{\min(j,k-j)+2} k^2{k \choose \ell}_q\right) \gamma \norm{f_\ell}
\]
\end{corollary}
\begin{proof}
This follows from combining \Cref{cor:balanced-eig}, \Cref{cor:canon-q-HDX-spectra}, and \Cref{prop:Grassmann-decomp}. Let $t=k-j$. In particular, it is sufficient to note that (in the notation of \Cref{cor:balanced-eig}):
\begin{align*}
\sum\limits_{Y \in \mathcal Y}\alpha_Y \lambda_{Y,\delta,\ell}&=\frac{1}{q^{(k-t)^2}{k \choose t}_q}\sum\limits_{i=0}^{k-t} (-1)^{k-t-i}q^{{k-t-i \choose 2}}{k-t \choose i}_q{k + i -\ell \choose i}_q\\
&= \frac{{k-j \choose \ell}_q}{{k \choose \ell}_q}.
\end{align*}
and further that:
\begin{align*}
w(S^j_k) &= \frac{1}{q^{j^2}{k \choose k-j}_q}\sum\limits_{i=0}^j q^{{j-i \choose 2}}{j \choose i}_q{k+i \choose i}_q\\
&\leq \frac{q^{jk}}{q^{j^2}{k \choose k-j}_q}\sum\limits_{i=0}^j q^{-i}\\
&\leq \left(\frac{q}{q-1}\right)^{\min(j,k-j)+1}
\end{align*}
\end{proof}
Again, since the swap walks are self-adjoint \Cref{thm:approx-ortho} implies that for small enough $\gamma$ the true spectra is closely concentrated around these values as well. It is worth noting that if the above analysis is repeated using the exact eposet parameters for the Grassmann (see \Cref{eq:grassmann-params}), this recovers the standard eigenvalues of the Grassmann graphs (see e.g. \cite{delsarte1976association}).


\subsection{Pseudorandom Functions and Small Set Expansion}
With an understanding of the spectra of HD-walks on $q$-simplicial complexes, we move to studying its combinatorial structure. By direct computation, it is not hard to show that on $q$-eposets, $\rho^k_\ell=\frac{1}{{k \choose \ell}_q}$ (\Cref{claim:reg} would only imply this is approximately true). As a result, we get a level-$i$ inequality for $q$-simplicial complexes that is the natural $q$-analog of BHKL's inequality for basic simplicial complexes.
\begin{theorem}\label{thm:grassmann-level}
Let $(X,\Pi)$ be a $\gamma$-$q$-eposet with $\gamma \leq q^{-\Omega(k^2)}$, and let $f: C_k \to \R$ be any function on $k$-faces with HD-Level-Set Decomposition $f= f_0+\ldots+f_k$. If $f$ is $(\varepsilon_1,\ldots,\varepsilon_\ell)$-$\ell_\infty$-pseudorandom, then for all $1 \leq i \leq \ell$:
\[
|\langle f, f_i \rangle| \leq  \left({k \choose i}_q + c\gamma\right)\varepsilon_i^2 + c\gamma\norm{f}^2.
\]
If $f$ additionally has $i$-local constant sign or is $(\varepsilon_1,\ldots,\varepsilon_\ell)$-$\ell_2$-pseudorandom, then
\[
|\langle f, f_i \rangle| \leq  {k \choose i}_q\varepsilon_i|\mathbb{E}[f]| + c\gamma\norm{f}^2
\]
where in both cases $c \leq q^{O(k^2)}$
\end{theorem}
For large enough $q,\gamma^{-1}$, this result is exactly tight. The key to showing this fact is to examine a local structure unique to the Grassmann called \textit{co-links}. The co-link of an element $W \in X(k')$, is all of the subspaces \textit{contained in} $W$:
\[
\bar{X}_W = \{ V \in X(k): V \subseteq W \}.
\]
Just like links, co-links of dimension $i$ (that is $k'=d-i$) also come through levels $0$ through $i$ of the complex, although this is somewhat trickier to see.
\begin{lemma}[HD-level-set decomposition of co-links]
Let $X=G_q(d,k)$ and ${\cal S}=\overline{X}_W$ be a co-link of dimension $i$ for $W\in X(d-i)$. Then, we have
\begin{align*}
    \id{\cal S}\in V_k^0\oplus \cdots \oplus V_k^{i}.
\end{align*}
\end{lemma}
\begin{proof}
Since we know that $V_k^0\oplus \cdots \oplus V_k^i=\mathrm{Im}(U_i^k C_i)$ (see e.g.\ \cite{dikstein2018boolean}), all we need to do is to show that there exists an $f\in C_i$ such that $\id{\cal S}=U_i^k f$. More specifically, we can write $f=\sum_{U\in X(i)}\alpha_U \id{U}$. Then, we have
\begin{align*}
    (U_i^k f)(V) = ~ \sum_{U\in X(i)} \alpha_U (U_i^k \id{U})(V)
    =~ \frac{1}{R(k,i)}\sum_{U\in X(i),U\subset V} \alpha_U.
\end{align*}
Suppose $\alpha_U=g(\dim(U\cap W))$ for some function $g:\{0,\dots,i\}\rightarrow \R$. We will prove that there exists a unique $g$ that satisfies the desired equations.

Consider the dimension of $V\cap W$. If $V\subset W$, i.e., $\dim(V\cap W)=k$, then for all $U \in X(i)$ s.t.\ $U\subset V$, $\dim(U\cap W)=i$. Then, for all $V\subset W$ we must have:
\begin{align*}
    U_i^k \id{V} = \frac{1}{R(k,i)}\sum_{U\in X(i),U\subset V} g(i) = g(i) = 1.
\end{align*}
On the other hand, consider $V\not\subset W$. In this case we must have $\dim(V\cap W)=k-j$ for some $i \geq j>0$ and further that $\dim(U\cap W)\in \{i-j,\dots, i\}$ for all $U \in X(i)$ s.t.\ $U\subset V$. This gives the following set of linear equations:
\begin{align*}
    U_i^k \id{V} = \sum_{\ell = i-j}^{i-1} c_{j,\ell} g(\ell) + c_{j,i} = 0 ~~~\forall 1\leq j \leq i,
\end{align*}
where $c_{j,\ell} := R(k,i)^{-1}\cdot |\{U\in X(i): U\subset V, \dim(U\cap W)=\ell, \dim(V\cap W)=k-j\}|$ is a constant for all $\ell\in \{i-j, \dots, i\}$. Since this system can be written as a triangular form with positive diagonal, it is invertible and there exists a unique solution for $g(0),\dots, g(i-1)$ as desired. By definition, such a solution must satisfy $f=\sum_{U\in X(i)}g(\dim(U\cap W))\id{U}$, so we have constructed $f \in C_i$ such that $U_i^k f=\id{\cal S}$, which completes the proof of the claim.
\end{proof}
Using this fact, we can show that our level-$i$ inequality is exactly tight.
\begin{proposition}
Let $X=G_q(d,k)$ be the Grassmann poset. For any $i \leq k \in \mathbb{N}$ and $c<1$, there exist large enough $q,d$ and a set $S \subset X(k)$ such that
\[
\langle \id{S}, \id{S,i} \rangle > c{k \choose i}_q\varepsilon_i\langle \id{S}, \id{S} \rangle
\]
where $S$ is $(i,\varepsilon_i)$-pseudorandom.
\end{proposition}
\begin{proof}
The proof goes through examining a ``co-link'' of dimension $i$, that is for $W \in X(d-i)$:
\[
\bar{X}_W = \{ V \in X(k) : V \subset W \}.
\]
For simplicity, let $S \coloneqq \bar{X}_W$.
The density of the co-link $S$ in any $j$-link $X_V$ is:
\[
\alpha_j = \frac{(q^{d-i-j}-1)\ldots (q^{d-k+1-i}-1)}{(q^{d-j}-1) \ldots (q^{d-k+1}-1)} = q^{-i(k-j)} + o_{q,d}(1).
\]
The idea is now to examine the (non)-expansion of the co-link with respect to the lower walk $U_{k-1}D_k$. By direct computation, the probability of returning to $\bar{X}_W$ after moving to a $(k-1)$-dimensional subspace is exactly:
\begin{equation}\label{eq:grassmann-exp}
\bar{\Phi}(\bar{X}_W) = \frac{q^{d-i}-q^{k-1}}{q^d-q^{k-1}} = q^{-i} \pm q^{-\Omega(d)}
\end{equation}
On the other hand, by \Cref{prop:pure-eig-vals} the approximate eigenvalues of the lower walk are given by
\[
\lambda_j = \frac{q^{k-j}-1}{q^{k}-1}= q^{-j} - O(q^{-k})
\]
Since a dimension-$i$ co-link has no projection onto levels $i+1$ through $k$, we can also write the non-expansion as:
\begin{align*}
\bar{\Phi}(\bar{X}_W) = \frac{1}{\langle \id{S},\id{S} \rangle}\sum\limits_{j=0}^i q^{-j} \langle \id{S},\id{S,j} \rangle - O(q^{-k})
\end{align*}
for large enough $q,d$. Combining this with our previous formula for the non-expansion in \Cref{eq:grassmann-exp}, we get that there exists a universal constant $c'$ such that for large enough $q$ and $d$, $\id{\bar{X}_{W}}$ cannot have more than a $\frac{c'}{q}$ fraction of its mass on levels $1$ through $i-1$. Finally, noticing that:
\[
{k \choose i}_q \alpha_i = 1 + o_q(1)
\]
we have
\begin{align*}
\frac{\langle \id{S}, \id{S,i}\rangle}{\langle \id{S}, \id{S}\rangle} \geq \frac{q-c'}{q}
\geq c{k \choose i}_q\alpha_i
\end{align*}
since the latter is strictly bounded away from $1$ for large enough $q$. This completes the result since $\bar{X}_W$ is $(\alpha_i,i)$-pseudorandom.
\end{proof}

We'll close the section by giving an immediate application of \Cref{thm:grassmann-level} to the expansion of pseudorandom sets, and briefly discuss connections with the proof of the 2-2 Games Conjecture and algorithms for unique games. Namely, as corollary of \Cref{thm:grassmann-level}, we show that for both the canonical and partial-swap walks, sufficiently pseudorandom functions expand near perfectly.

\begin{corollary}[$q$-eposets Edge-Expansion]\label{cor:grassmann-exp}
Let $(X,\Pi)$ be a $d$-dimensional $\gamma$-$q$-eposet, $S \subset X(k)$ a subset whose indicator function $\id{S}$ is $(\varepsilon_1,\ldots,\varepsilon_\ell)$-pseudorandom. Then the edge expansion of $S$ with respect to the canonical walk $N_k^j$ is bounded by:
\[
\Phi_{\pi_k}(N^j_k,S) \geq 1 - \mathbb{E}[\id{S}] - \sum\limits_{i=1}^\ell \frac{{k+j-i \choose j}_q}{{k+j \choose j}_q}{k \choose i}_q\varepsilon_i  - q^{-(\ell+1)j} - q^{O(k^2)}\gamma
\]
Further, the edge expansion of $S$ with respect to the partial-swap walk $S_k^j$ is bounded by:
\[
\Phi_{\pi_k}(S^j_k,S) \geq 1 - \mathbb{E}[\id{S}] - \sum\limits_{i=1}^\ell {k-j \choose i}_q \varepsilon_i - q^{-(\ell+1)j} - q^{O(k^2)}\gamma
\]
\end{corollary}

Note that $S^j_k$ on $q$-eposets is a generalization of the Grassmann Graphs (and are equivalent when $X$ is the Grassmann Poset). While our definition of pseudorandomness is weaker than that of \cite{subhash2018pseudorandom} and therefore necessarily depends on the dimension $k$, we take the above as evidence that the framework of expanding posets may be important for making further progress on the Unique Games Conjecture. In particular, combined with recent works removing this $k$-dependence on simplicial complexes \cite{bafna2021hypercontractivity,gur2021hypercontractivity}, it seems plausible that the framework of expanding posets may lead to a more general understanding of the structure underlying the unique games conjecture.
\bibliographystyle{unsrtnat}  
\bibliography{references} 
\newpage
\appendix
\section{Eposet Parameters and Regularity}\label{app:regularity}
In this appendix we will discuss connections between notions of regularity, the averaging operators, and eposet parameters. To start, we'll show that downward and middle regularity (which are defined only on adjacent levels of the poset) imply extended regularity between any two levels.
\begin{proposition}
Let $(X,\Pi)$ be a $d$-dimensional regular measured poset. Then for any $i < k \leq d$, there exist regularity constant $R(k,i)$ such that for any $x_k \in X(k)$, there are exactly $R(k,i)$ elements $x_i \in X(i)$ such that $x_k > x_i$.
\end{proposition}
\begin{proof}
Given any element $x_k \in X(k)$, downward regularity promises there are exactly $\prod_{j=i+1}^k R(j)$ unique chains $x_k < x_{k-1} < \ldots < x_{i+1} < x_i$. By middle regularity, any fixed $x_i \in X(i)$ which appears in this fashion appears in exactly $m(k,i)$ chains. Noting that $x_i < x_k$ if and only if $x_i$ appears in such a chain, the total number of $x_i < x_k$ must be exactly:
\[
R(k,i) = \frac{\prod_{j=i+1}^k R(j)}{m(k,i)}.
\]
\end{proof}
A similar argument shows that regularity allows the up operators to compose in the natural way.
\begin{proposition}
Let $(X,\Pi)$ be a $d$-dimensional regular measured poset. Then for any $i < k \leq d$ we have:
\[
U^k_if(x_k) = \frac{1}{R(k,i)}\sum\limits_{x_i<x_k}f(x_i)
\]
\end{proposition}
\begin{proof}
Expanding out $U^k_i f(y)$ gives:
\begin{align*}
    U^k_i f(x_k) = \frac{1}{\prod\limits_{j=i+1}^kR(j)} \sum\limits_{x_{k-1}<x_k}\ldots\sum\limits_{x_{i}<x_{i+1}} f(x_i)
\end{align*}
The number of times each $x_i$ appears in this sum is exactly the number of chains starting at $x_k$ and ending at $x_i$, so by middle regularity:
\begin{align*}
    \frac{1}{\prod\limits_{j=i+1}^k R(j)} \sum\limits_{x_{k-1}<x_k}\ldots\sum\limits_{x_{i+1}<x_i} f(x_i) &= \frac{m(k,i)}{\prod\limits_{j=i+1}^k R(j)}\sum\limits_{x_{i}<x_{k}} f(x_i)\\
    &=\frac{1}{R(k,i)}\sum\limits_{x_{i}<x_{k}} f(x_i).
\end{align*}
as desired.
\end{proof}

We'll now take a look at the connection between eposet parameters and regularity. It is convenient to first start with a lemma stating that non-laziness is equivalent to bounding the maximum transition probability of the lower walk.
\begin{lemma}\label{lemma:anti-concentration}
Let $(X,\Pi)$ be a $d$-dimensional measured poset. Then for any $0 < i \leq d$, the maximum laziness of the lower walk is also the maximum transition probability:
\[
\max_{\sigma \in X(i)}\left\{ \id{\sigma}^T U_{i-1}D_i \id{\sigma}  \right\} = \max_{\sigma,\tau \in X(i)}\left\{\id{\sigma}^T U_{i-1}D_i\id{\tau}\right\}.
\]
\end{lemma}
\begin{proof}
Assume that $\tau \neq \sigma$. Then the transition probability from $\tau$ to $\sigma$ is exactly
\begin{align*}
\id{\sigma}^T U_{i-1}D_i\id{\tau} &= \frac{\pi_\tau(\sigma \setminus \tau)}{R(i,i-1)}\\
& \leq \frac{1}{R(i,i-1)}\sum\limits_{\tau \lessdot \sigma} \pi_\tau(\sigma \setminus \tau)\\
&= \id{\sigma}^\tau U_{i-1}D_i\id{\sigma},
\end{align*}
which implies the result.
\end{proof}
We now prove our two claims relating the eposet parameters to regularity.
\begin{claim}
Let $(X,\Pi)$ be a $d$-dimensional $(\delta,\gamma)$-eposet. 
Then for every $1 \leq k \leq d$
and $0 \leq i \leq k$, the following approximate relation between the eposet and regularity parameters holds:
\[
\lambda_i(N_k^1) \in  \frac{R(k,i)}{R(k+1,i)} \pm \left(\gamma^k_{k-i} + R(k,i)\delta^k_{k-i}\gamma\right)
\]
where we recall $\lambda_i(N_k^1)=1-\prod\limits^{k}_{j=i} \delta_j$.
\end{claim}
\begin{proof}
One of our main analytical tools so far has been the relation between the upper and lower walks given in  \Cref{lemma:DU-UD}:
\[
\norm{D_{k+1}U^{k+1}_{i} - (1-\delta^k_{k-i})U^{k}_{i} - \delta_{k-i}^k U^{k}_{i-1}D_{i}} \leq \gamma^k_{k-i}.
\]
For this result, we'll actually need a refinement of this result given in \cite[Lemma A.1]{bafna2020high}:\footnote{Formally the result is only stated for simplicial complexes in \cite{bafna2020high}, but the same proof holds for eposets.}
\begin{equation}\label{eq:UD-DU-refined}
D_{k+1}U^{k+1}_{i} = (1-\delta^k_{k-i})U^{k}_{i} + \delta_{k-i}^k U^{k}_{i-1}D_{i} + \sum\limits_{j=-1}^{k-i-1}U^k_{k-j-1}\Gamma_jU^{k-j-1}_{i}
\end{equation}
where $\sum \norm{\Gamma_j} \leq \gamma^k_{k-i}$. The idea is now to examine the ``laziness'' of the two sides of this equality. In other words, given a starting $k$-face $\tau$, what is the probability that the resulting $i$-face $\sigma$ satisfies $\sigma < \tau$? 

To start, we'll argue that the laziness of the lefthand side is exactly $\frac{R(k,i)}{R(k+1,i)}$. This follows from noting that there are $R(k,i)$ $i$-faces $\sigma$ satisfying $\sigma < \tau$, and $R(k+1,i)$ options after taking the initial up-step of the walk to $\tau' > \tau$. After the down-steps, the resulting $i$-face is uniformly distributed over these $R(k+1,i)$ options $\sigma < \tau'$, and since every $\sigma < \tau < \tau'$, all original $R(k,i)$ lazy options are still viable after the up-step to $\tau'$.

Analyzing the right-hand side is a bit trickier. The initial term $(1-\delta^k_{k-i})U^k_i$ is completely lazy, so it contributes exactly $(1-\delta^k_{k-i})=\lambda_i(N_k^1)$. We'll break the second term into two steps: walking from $X(k)$ to $X(i)$ via $U^k_i$, then from $X(i)$ to $X(i)$ via the lower walk $U_{i-1}D_i$. Starting at a $k$-face $\tau$, notice that after applying the down step $U^k_i$ we are uniformly spread over $\sigma < \tau$. Computing the laziness then amounts to asking what the probability of staying in this set is after the application of $UD$, which one can naively bound by the maximum transition probability times the set size $R(k,i)$. By non-laziness, the maximum transition probability is at most $\gamma$ (see \Cref{lemma:anti-concentration}).

The third term can be handled similarly. The first down step $U^k_{k-j-1}$ spreads $\tau$ evenly across $\sigma < \tau$ in $X(k-j-1)$. The resulting $i$-face $\sigma'$ after applying $\Gamma_j U^{k-j-1}_i$ is less than $\tau$ if and only if the intermediary $(k-j-1)$-face after applying $\Gamma_j$ is less than $\tau$, which is bounded by the spectral norm $\norm{\Gamma_j}$.\footnote{We note that $\Gamma_j$ is not stochastic, but it is self-adjoint and an easy exercise to see that the analogous reasoning still holds.}

Putting everything together, since both sides of \Cref{eq:UD-DU-refined} must have equivalent laziness, we get that $\lambda_i(N_k^1)$ must be within $\sum \norm{\Gamma_j} + \delta^k_{k-i}R(k,i)\gamma$ as desired.
\end{proof}
\Cref{claim:reg} and \Cref{claim:reg-lower} can both be proving an analogous theorem for the upper walk.
\begin{claim}[Regularity and Upper Walk Spectrum]
Let $(X,\Pi)$ be a $d$-dimensional $(\delta,\gamma)$-eposet. Then for any $j \leq i \leq k < d$, we have:
\[
\lambda_j(N^{k-i}_i) \in \frac{R(i,j)}{R(k,i)} \pm err,
\]
where $err \leq O\left(\frac{i^4k^2R_{\text{max}}}{\delta_i(1-\delta_{i-1})}\gamma\right)$.
\end{claim}
\begin{proof}
This follows almost immediately from the fact that $i$-links lie almost entirely on the $i$th eigenstrip (\Cref{lemma:link-projection}). In particular, it is enough to examine the expansion of $i$-links with respect to the upper canonical walk $N_i^{k-i}$. On the one hand, for any $j \leq i$ and $\tau \in X(j)$ we have:
\begin{align*}
\bar{\Phi}(X^i_\tau) &= \frac{\langle \id{X^i_\tau}, N^{k-i}_i\id{X^i_\tau}\rangle}{\langle \id{X^i_\tau}, \id{X^i_\tau}\rangle}\\
&=\frac{\langle U^k_j\id{\tau}, U^k_j\id{\tau} \rangle}{\langle U^i_j\id{\tau},U_j^i\id{\tau} \rangle}\\
&= \frac{R(i,j)^2}{R(k,i)^2}\frac{\langle \id{X^k_\tau},\id{X^k_\tau} \rangle}{\langle \id{X^i_\tau},\id{X^i_\tau} \rangle}\\
&= \frac{R(i,j)}{R(k,i)}\frac{\langle \id{\tau},\id{\tau} \rangle}{\langle \id{\tau},\id{\tau} \rangle}\\
&= \frac{R(i,j)}{R(k,i)}.
\end{align*}
where we have applied the fact that $\langle X^\ell_\tau,X^\ell_\tau \rangle = R(\ell,j)\langle \id{\tau},\id{\tau}\rangle$. On the other hand, by \Cref{lemma:link-projection} we also have that:
\begin{align*}
    \bar{\Phi}(\id{X^i_\tau}) &= \frac{1}{\langle \id{X^i_\tau},\id{X^i_\tau} \rangle}\sum\limits_{\ell=0}^i \langle \id{X^i_\tau},N^{k-i}_i\id{X^i_\tau,\ell} \rangle\\
    & \in \frac{1}{\langle \id{\tau},\id{\tau} \rangle}\sum\limits_{\ell=0}^i \lambda_j(N^{k-i}_i)\langle \id{X^i_\tau},\id{X^i_\tau,\ell} \rangle + c\gamma \\
    & \in \lambda_j(N^{k-i}_i)\frac{\langle \id{\tau},\id{\tau,j} \rangle}{\langle \id{\tau},\id{\tau} \rangle}+\sum\limits_{j=0}^i err_1\\
    & \in \lambda_j(N^{k-i}_i) + err_2
\end{align*}
where as in the proof of \Cref{thm:link-lower}, $c,err_1,err_2 \leq O\left(\frac{i^4k^2R_{\text{max}}}{\delta_{i-j}^i(1-\delta_{j-1})}\gamma\right)$.
\end{proof}
\Cref{claim:reg} follows immediately from observing that $\rho^k_i = \lambda_i(N^{k-i}_i)$ (by \Cref{prop:pure-eig-vals}). \Cref{claim:reg-lower} follows from observing that $\widehat{N}_i^{k-i}$ and $\widecheck{N}_k^{k-i}$ have the same approximate eigenvalues (similarly by \Cref{prop:pure-eig-vals}).

Finally we close out the section by discussing the connection between non-laziness and a variant of eposets called local-spectral expanders \cite{kaufman2021local}. To start, let's recall this latter definition. 
\begin{definition}[Local-Spectral Expander \cite{dinur2017high,kaufman2021local}]
A $d$-dimensional measured poset $(X,\Pi)$ is a $\gamma$-local-spectral expander if the graph underlying every link\footnote{Here the link of $\tau$ is not just its top level faces, but the complex given by taking this set, removing $\tau$ from each face, and downward closing.} of dimension at most $d-2$ is a $\gamma$-spectral expander.\footnote{A graph is a $\gamma$-spectral expander if its weighted adjacency matrix has no non-trivial eigenvalues greater than $\gamma$ in absolute value.}
\end{definition}
Under suitable regularity conditions (see \cite{kaufman2021local}), local-spectral expansion is equivalent to the notion of expanding posets used in this work. A simple argument shows that $\gamma$-local-spectral expanders are $\gamma$-non-lazy.
\begin{lemma}
Let $(X,\Pi)$ be a $d$-dimensional $\gamma$-local-spectral expander, and $0 < i < d$. The laziness of the lower walk on level $i$ is at most:
\[
\max_{\sigma \in X(i)}\left\{\frac{\langle \id{\sigma}, U_{i-1}D_i \id{\sigma} \rangle}{\langle \id{\sigma}, \id{\sigma} \rangle} \right\} \leq \gamma.
\]
\end{lemma}
\begin{proof}
Through direct computation, the laziness probability of the lower walk at $\sigma \in X(i)$ is exactly
\[
\frac{\langle \id{\sigma}, U_{i-1}D_i \id{\sigma} \rangle}{\langle \id{\sigma}, \id{\sigma} \rangle} = \frac{1}{R(i,i-1)} \sum \limits_{\tau \lessdot \sigma } \pi_\tau(\sigma \setminus \tau)
\]
It is therefore enough to argue that $\pi_\tau(\sigma \setminus \tau) \leq \gamma$. This follows from the fact that the graph underlying the link $X_\tau$ is a $\gamma$-spectral expander. In particular, recall that an equivalent formulation of this definition states that:
\[
\norm{A_\tau - UD_\tau} \leq \gamma, 
\]
where $A_\tau$ is the standard (non-lazy upper) walk and $UD_\tau$ is the lower walk on the graph underlying $X_\tau$. This implies that the weight of any vertex $v$ in the graph is at most $\gamma$, as:
\[
\frac{\langle \id{v},UD_\tau \id{v}\rangle}{\langle \id{v}, \id{v} \rangle} = \frac{\langle \id{v},(UD_\tau - A_\tau)\id{v}\rangle}{\langle \id{v}, \id{v} \rangle} \leq \norm{A_\tau - UD_\tau} \leq \gamma
\]
where we have used the fact that $A_\tau$ is non-lazy by definition. Since $\pi_\tau(\sigma \setminus \tau)$ is exactly the weight of the vertex $\sigma \setminus \tau$ in $X_\tau$, this completes the proof.
\end{proof}
\end{document}